\documentclass[a4paper]{amsart}
\usepackage[english]{babel}
\usepackage[utf8]{inputenc}
\usepackage{amssymb,graphicx,enumerate,hyperref,latexsym,float}

\newcommand{\TeXmacs}{T\kern-.1667em\lower.5ex\hbox{E}\kern-.125emX\kern-.1em\lower.5ex\hbox{\textsc{m\kern-.05ema\kern-.125emc\kern-.05ems}}}
\newcommand{\assign}{:=}
\newcommand{\barsuchthat}{|}
\newcommand{\citewebsite}{This document has been written using GNU {\TeXmacs}; see \url{https://www.texmacs.org}.}
\newcommand{\divides}{\mathrel{|}}
\newcommand{\exterior}{\wedge}
\newcommand{\longdownarrow}{{\mbox{\rotatebox[origin=c]{-90}{$\longrightarrow$}}}}
\newcommand{\longdownequal}{{\mbox{\rotatebox[origin=c]{-90}{$\longequal$}}}}
\newcommand{\longequal}{{=\!\!=}}
\newcommand{\longrightarrowlim}{\mathop{\longrightarrow}\limits}
\newcommand{\longuparrow}{{\mbox{\rotatebox[origin=c]{90}{$\longrightarrow$}}}}
\newcommand{\ndivides}{\mathrel{\nmid}}
\newcommand{\of}{:}
\newcommand{\tmdfn}[1]{\textit{#1}}
\newcommand{\tmem}[1]{{\em #1\/}}
\newcommand{\tmnote}[1]{\thanks{\textit{Note:} #1}}
\newcommand{\tmop}[1]{\ensuremath{\operatorname{#1}}}
\newcommand{\tmscriptoutput}[4]{#4}
\newcommand{\tmtextit}[1]{\text{{\itshape{#1}}}}
\newenvironment{enumeratealpha}{\begin{enumerate}[a{\textup{)}}] }{\end{enumerate}}
\newenvironment{enumeratenumeric}{\begin{enumerate}[1.] }{\end{enumerate}}
\newenvironment{enumerateroman}{\begin{enumerate}[i.] }{\end{enumerate}}
\newenvironment{itemizedot}{\begin{itemize} }{\end{itemize}}
\newenvironment{proof*}[1]{\noindent\textbf{#1\ }}{\hspace*{\fill}$\Box$\medskip}
\newtheorem{corollary}{Corollary}[section]
{\theoremstyle{definition}\newtheorem{definition}[corollary]{Definition}}
{\theoremstyle{remark}\newtheorem{example}[corollary]{Example}}
\newtheorem{lemma}[corollary]{Lemma}
\newtheorem{notation}[corollary]{Notation}
\newtheorem{proposition}[corollary]{Proposition}
{\theoremstyle{remark}\newtheorem{question}{Question}}
{\theoremstyle{remark}\newtheorem{remark}[corollary]{Remark}}
\newtheorem{theorem}[corollary]{Theorem}
{\theoremstyle{remark}\newtheorem{warning}[corollary]{Warning}}

\begin{document}
\sloppy

\title{Perfectoid rings as Thom spectra}
\tmnote{{\citewebsite}}

\author{Zhouhang Mao}

\date{May 30, 2021}

\begin{abstract}
  The Hopkins-Mahowald theorem realizes the Eilenberg-Maclane spectra
  $H\mathbb{F}_p$ as Thom spectra for all primes $p \in \mathbb{N}_{> 0}$. In
  this article, we record a known proof of a generalization of
  Hopkins-Mahowald theorem, realizing $H k$ as Thom spectra for perfect rings
  $k$, and we provide a further generalization by realizing $H R$ as Thom
  spectra for perfectoid rings $R$. We also discuss even further
  generalizations to prisms $(A, I)$ and indicate how to adapt our proofs to
  Breuil-Kisin case.
\end{abstract}

{\maketitle}

{\tableofcontents}

\section{Introduction}

In this article, since most of our results are $p$-typical, we fix a prime $p
\in \mathbb{N}_{> 0}$. We first describe the classical Hopkins-Mahowald
theorem. We know that $\mathbb{F}_p \cong \mathbb{Z}_p / p$, that is to say,
$\mathbb{F}_p$ is the free $\mathbb{Z}_p$-algebra in which $p = 0$. For some
reasons, we need to extend this kind of results to a category of ``less
linear'' algebras in which the addition is not commutative or even associative
on the nose, but only up to coherent homotopy. To be more precise, we need to
understand whether the (Eilenberg-Maclane) ring spectrum $H\mathbb{F}_p$ is
still the free object in the category of $\mathbb{S}_p^{\wedge}$-algebras
satisfying certain associativity and commutativity with $p = 0$? The classical
Hopkins-Mahowald theorem answers this affirmatively: they are the free object
in the category of $\mathbb{E}_2$-$\mathbb{S}_p^{\wedge}$-algebras with $p =
0$. There are two ways to describe ``free $\mathbb{E}_2$-algebras with $p =
0$''. In this article, we will mainly adopt the description via Thom spectra.
We will go to another, more direct and natural but technically more burdened
description in Section~\ref{sec:thom-free-E2-algs}. We start with formal
definitions of Thom spectra with informal illustrations and refer to
{\cite{Antolin-Camarena2014}} for further discussions.

\begin{definition}
  Given a ring spectrum $R$, we define the $\infty$-category $\tmop{BGL}_1
  (R)$ to be the full subcategory of $\tmop{LMod}_R^{\simeq}$ spanned by left
  $R$-module spectra equivalent to $R$, where we denote by
  $\mathcal{C}^{\simeq}$ the maximal groupoid associated to an
  $\infty$-category $\mathcal{C}$. 
\end{definition}

\begin{remark}
  \label{rem:bgl1mongrpd}The $\infty$-category $\tmop{BGL}_1 (R)$ is in fact
  an $\infty$-groupoid, and if we further suppose that $R$ is an
  $\mathbb{E}_{n + 1}$-ring spectrum, then $\tmop{BGL}_1 (R)$ inherits an
  $\mathbb{E}_n$-monoidal structure from $\tmop{LMod}_R$.
\end{remark}

We admit the following result, which could be understand as an analogue of the
fact that $\pi_1 (B G) = G$ for any discrete group $G$:

\begin{proposition}
  $\pi_1 (\tmop{BGL}_1 (R)) = \tmop{GL}_1 (\pi_0 R)$ for any ring spectra $R$.
  Concretely, an invertible element $a \in \pi_0 R$ corresponds to a
  multiplication map $m_a \of R \rightarrow R$ in $\tmop{BGL}_1 (R)$.
\end{proposition}

\begin{remark}
  In fact, $\tmop{BGL}_1 (R)$ is a delooping of the group of invertible
  elements in $R$. 
\end{remark}

Now we recall the definition of Thom spectra:

\begin{definition}
  \label{def:thomsp}Given a ring spectrum $R$, a space $X$ and a map $f : X
  \rightarrow \tmop{BGL}_1 (R)$, {\tmdfn{the Thom spectrum $M f$ associated to
  $f$}} is the colimit of the composition
  \[ X \rightarrow \tmop{BGL}_1 (R) \rightarrow \tmop{LMod}_R \]
\end{definition}

We note that by definition of colimits, we can understand the colimit as a
kind of ``free objects satisfying several equations''. We will choose a
special space $X$ to encode the $\mathbb{E}_2$-commutativity (understood as a
generalized version of classical associativity, a collection of equations) and
a map $f \of X \rightarrow \tmop{BGL}_1 (R)$ to encode the ``equation'' $p =
0$.

\begin{remark}
  As a special case of {\cite[Proposition~4.1.2.6]{Lurie2009}}, any homotopy
  equivalence of Kan complexes is cofinal, therefore the formation of the
  colimit does not depend on the choice of models of the space $X$.
\end{remark}

\begin{remark}
  In this article, we only consider the case that $R$ is a connective
  $\mathbb{E}_{\infty}$-ring spectrum. As a consequence, we can replace
  $\tmop{LMod}_R$ by $\tmop{Mod}_R$ and the Thom spectrum $M f$ is connective.
\end{remark}

\begin{remark}
  In Definition~\ref{def:thomsp}, if $X$ is endowed with an
  {\tmem{$\mathbb{E}_n$}}-algebra structure, and $f$ is assumed to be
  $\mathbb{E}_n$-monoidal, then the Thom spectrum $M f$ naturally inherits an
  $\mathbb{E}_n$-$R$-algebra structure. In this case, we will call $M f$ the
  $\mathbb{E}_n$-Thom spectrum associated to $f$.
\end{remark}

In the classical Hopkins-Mahowald theorem, we will choose $X = \Omega^2 S^3$,
the free $\mathbb{E}_2$-group in the $\infty$-category $\mathcal{S}$ of
spaces.

\begin{remark}
  \label{cons:f_p}As a special case, $\pi_1 (\tmop{BGL}_1
  (\mathbb{S}_p^{\wedge})) = \tmop{GL}_1 (\mathbb{Z}_p) = \{ a \in
  \mathbb{Z}_p \barsuchthat a \tmop{mod} p \neq 0 \}$. The invertible element
  $1 - pu$ in $\mathbb{Z}_p$ gives rise to a map $S^1 \rightarrow \tmop{BGL}_1
  (\mathbb{S}_p^{\wedge})$ where $u \in \tmop{GL}_1 (\mathbb{Z}_p)$ is an
  invertible element in $\mathbb{Z}_p$. Since the $p$-adic sphere spectrum
  $\mathbb{S}_p^{\wedge}$ is an $\mathbb{E}_{\infty}$-ring spectrum, by
  Remark~\ref{rem:bgl1mongrpd} this map extends to a double loop map $\Omega^2
  S^3 \simeq \Omega^2 \Sigma^2 S^1 \rightarrow \tmop{BGL}_1
  (\mathbb{S}_p^{\wedge})$, which we denote by $f_{\mathbb{F}_p, pu}$.
\end{remark}

We note that the choice of $1 - pu$ essentially imposes an equation $1 - pu =
1$. This could be seen by the fact that taking the colimit along
$f_{\mathbb{F}_p, pu}$ is essentially taking the homotopy orbits of the
$\Omega^2 S^3$-action, which is somehow ``multiplying by'' $1 - pu$.

\begin{remark}
  In the first drafts of this article, we simply took $u = 1.$ Later, we
  realized that it might be easier to introduce $u$ to fix a gap in
  commutative algebra for technical reasons.
\end{remark}

Now we formulate the classical Hopkins-Mahowald theorem
(cf.~{\cite[Theorem~5.1]{Antolin-Camarena2014}}, where $u = 1$, but the proof
works for the general case. See also {\cite[Theorem~A.1]{Krause}}):

\begin{theorem}[Hopkins-Mahowald]
  \label{thm:hopkins-mahowald}The Eilenberg-Maclane spectrum $H\mathbb{F}_p$
  is the $\mathbb{E}_2$-Thom spectrum associated to the map $f_{\mathbb{F}_p,
  pu} \of \Omega^2 S^3 \rightarrow \tmop{BGL}_1 (\mathbb{S}_p^{\exterior})$.
\end{theorem}

This arouses a natural question: what other discrete rings are Thom spectra in
a similar fashion? The first guess will come from the observation that
$\mathbb{Z}_p \cong W (\mathbb{F}_p)$, so it would be natural to ask whether
we have similar results for perfect $\mathbb{F}_p$-algebras?

In this article, our main purpose is to show that this is the case for
perfectoid rings (which is inspired by computational results of topological
Hochschild homology of perfectoid rings in {\cite{Bhatt2018}}), and
consequently, for perfect $\mathbb{F}_p$-algebras. In order to do so, we need
the concept of spherical Witt vectors $W^+ (k)$ for perfect
$\mathbb{F}_p$-algebras $k$, which we will recall in section~\ref{sec:ha}. For
the moment, we will take advantage of the fact that $\pi_0 (W^+ (k)) = W (k)$
where $W (k)$ is the ring of (classical) Witt vectors. One example is that
$W^+ (\mathbb{F}_p) \simeq \mathbb{S}_p^{\wedge}$.

\begin{remark}
  \label{cons:f_R}Given a perfectoid ring $R$, denote by $\xi$ a generator of
  the kernel of Fontaine's pro-infinitesimal thickening $\theta \of W
  (R^{\flat}) \rightarrow R$, which we will review in section~\ref{sec:perfd}.
  As in Remark~\ref{cons:f_p}, the invertible element in $W (R^{\flat})$, $1 -
  \xi \in \tmop{GL}_1 (W (R^{\flat})) = \pi_1 (\tmop{BGL}_1 (W^+
  (R^{\flat})))$ gives rise to a map $S^1 \rightarrow \tmop{BGL}_1 (W^+
  (R^{\flat}))$ which extends to a double loop map $f_{R, \xi} \of \Omega^2
  S^3 \rightarrow \tmop{BGL}_1 (W^+ (R^{\flat}))$.
\end{remark}

\begin{theorem}[Main Theorem]
  \label{thm:main}The Eilenberg-Maclane spectrum $H R$ is the
  $\mathbb{E}_2$-Thom spectrum associated to the map $f_{R, \xi}$ for any
  perfectoid ring $R$.
\end{theorem}

Fontaine's pro-infinitesimal thickening $\theta$ is in fact surjective. Note
that $R \cong W (R^{\flat}) / \xi$, and our result is amount to say that the
ring spectrum $H R$ is a free $\mathbb{E}_2$-$W^+ (R^{\flat})$-algebra with
$\xi = 0$.

\begin{remark}
  When $R$ is a perfect $\mathbb{F}_p$-algebra, we can take $\xi = pu$ where
  $u \in \tmop{GL}_1 (R)$ is an invertible element in $R$, and we note that
  $R^{\flat} = R$. Especially, when $R =\mathbb{F}_p$, $f_{R, pu}$ coincides
  with $f_{\mathbb{F}_p, pu}$, hence Theorem~\ref{thm:main} generalizes
  Theorem~\ref{thm:hopkins-mahowald}.
\end{remark}

\begin{remark}
  The composite map $W^+ (R^{\flat}) \xrightarrow{\tau_{\leq 0}} H W
  (R^{\flat}) \xrightarrow{H \theta} H R$ should be understood as a spherical
  analogue of Fontaine's map $\theta \of W (R^{\flat}) \rightarrow R$. We will
  establish a universal property,
  Proposition~\ref{prop:sph-Fontaine-inf-thickening}, similar to Fontaine's,
  Proposition~\ref{prop:Fontaine-inf-thickening}, which might be of
  independent interest.
\end{remark}

The motivation to realize $H\mathbb{F}_p$ as a free $\mathbb{E}_2$-algebra
with $p = 0$ is that it describes a direct ``generation-relation'' like
description with respect to the (p-completed) sphere spectrum
$\mathbb{S}_p^{\wedge}$. Similarly, realization of $H R$ as a free
$\mathbb{E}_2$-$W^+ (R)$-algebra with $\xi = 0$ enables us to relate $H R$
more directly to the ring $W^+ (R^{\flat})$ of spherical Witt vectors, which
allows us to deduce ``topological'' results about these rings. For example, as
a consequence, we can compute the topological Hochschild homology $\tmop{THH}
(H R)$ (of a perfectoid ring $R$) as an $\mathbb{E}_1$-ring spectrum and
deduce Bökstedt's periodicity. By {\cite[Proposition~4.7]{Krause}}, as in the
proof of Theorem~4.1 there, we have

\begin{proposition}
  \label{prop:thh-over-sphwitt}The (relative) topological Hochschild homology
  $\tmop{THH} (H R / W^+ (R^{\flat})) \simeq H R \otimes \Omega S^3$ as
  $\mathbb{E}_1$-$W^+ (R^{\flat})$-algebras for any perfectoid ring $R$.
\end{proposition}

The proof is somehow technical, but essentially it is similar to the classical
computation of the Hochschild homology $\tmop{HH} (R / W (R^{\flat}))$, via
resolving $R$ by $W (R^{\flat})$-CDGAs. We refer to first paragraphs of the
proof of {\cite[Theorem~1.3.2]{Hesselholt2019}} for this classical case. As a
consequence of Proposition~\ref{prop:thh-over-sphwitt}, we have (see
subsection~\ref{subsec:THH}):

\begin{proposition}
  \label{prop:compl-thh}The (absolute) topological Hochschild homology
  $\tmop{THH} (H R)_p^{\wedge} \simeq H R \otimes \Omega S^3$ as
  $\mathbb{E}_1$-ring spectra.
\end{proposition}

By known results on the homology of $\Omega S^3$ (a classical reference is
{\cite{Bott1982}}), we deduce Bökstedt's periodicity for perfectoid rings
(cf.~{\cite[Theorem~6.1]{Bhatt2018}}).

\begin{corollary}[Bökstedt's periodicity]
  $\pi_{\ast} (\tmop{THH} (H R)_p^{\wedge}) \cong R [u]$ where $u$ is any
  generator of $\pi_2 (\tmop{THH} (H R)_p^{\wedge})$ as a $\pi_0 (\tmop{THH}
  (H R)_p^{\wedge})$-module.
\end{corollary}

In fact, our question was motivated by Bökstedt's periodicity for perfectoid
rings: we wanted to understand why Bökstedt's periodicity holds.

Further generalizations of Theorem~\ref{thm:main} to prisms, the concept
introduced in {\cite{Bhatt2019}}, seem plausible. However, we are only capable
to reach another special case of prisms motivated by Breuil-Kisin cohomology,
parallel to the perfectoid case, proposed by Matthew Morrow:

\begin{theorem}
  Let $A$ be complete discrete valuation ring of mixed characteristic with
  residue field $k$ being perfect of characteristic $p$. Then the
  Eilenberg-Maclane spectrum $H A$ is the $\mathbb{E}_2$-Thom spectrum
  associated to a map $f_E \of \Omega^2 S^3 \rightarrow \tmop{BGL}_1 (W^+ (k)
  [[u]])$.
\end{theorem}

Inspired by {\cite[Section~9]{Krause2019}}, we will also provide a version of
Hopkins-Mahowald theorem for complete regular local rings:

\begin{theorem}
  Let $(A, \mathfrak{m})$ be a complete regular local ring of mixed
  characteristic with residue field $k$ being perfect of characteristic $p$.
  Let $(a_1, \ldots, a_n) \subseteq \mathfrak{m}$ be a regular sequence which
  generates the maximal ideal $\mathfrak{m}$. Then the Eilenberg-Maclane
  spectrum $H A$ is the $\mathbb{E}_2$-Thom spectrum associated to a map $f_A
  \of \Omega^2 S^3 \rightarrow \tmop{BGL}_1 (W^+ (k) [[u_1, \ldots, u_n]])$.
\end{theorem}

In this article, we will first review spherical Witt vectors. We then record a
known proof of perfect rings being Thom spectrum, the special case of
Theorem~\ref{thm:main} for perfect rings, which we learn from Sanath
Devalapurkar, but the proof is also well-known to experts such as Achim Krause
and Thomas Nikolaus, see {\cite{Krause2019}}. This result is needed in the
proof of the general case of Theorem~\ref{thm:main}. Then we start with
recalling the definition and some basic properties of perfectoid rings, and
prove Theorem~\ref{thm:main}. As far as we know, although this is known to
several experts, the proof is not found in the literature. We will finally
discuss further generalizations to prisms in Section~\ref{sec:analog}, and
especially Hopkins-Mahowald theorem for Breuil-Kisin cases, which seems also
to be known by experts (see {\cite[Remark~3.4]{Krause2019}}). We take an
opportunity to write down those proofs. The author thanks Matthew Morrow for
various suggestions during the construction of this article.

\begin{warning}
  For spectra $M, N$, we will denote the smash product of $M, N$ by $M \otimes
  N$. Let $R$ be an $\mathbb{E}_1$-ring (spectrum), $M$ a right $R$-module
  (spectrum) and $N$ a left $R$-module (spectrum), we will denote the relative
  tensor product by $M \otimes_R N$. In order to avoid possible ambiguities,
  for discrete rings $A$, right $A$-modules $P$ and left $A$-modules $Q$, we
  will denote the ordinary (algebraic) tensor product by $\tmop{Tor}_0^A (P,
  Q)$ (instead of $P \otimes_A Q$). It is important that in general the
  Eilenberg-Maclane spectrum $H \tmop{Tor}_0^A (P, Q)$ do not coincide with
  the relative tensor product $H P \otimes_{H A} H Q$ of spectra. Rather, the
  relative tensor $H P \otimes_{H A} H Q$ coincides with the Eilenberg-Maclane
  spectrum $H (P \otimes_A^{\mathbb{L}} Q)$ of the derived tensor product.
  Since the concept of the derived tensor product does not play a great role
  in this article, we will not use the notation $\otimes_A^{\mathbb{L}}$, and
  we will uniformly preserve the notation $\otimes$ for smash products and
  relative tensor products of spectra.
\end{warning}

\begin{notation}
  In this article, we mainly adopt notations in {\cite{Lurie2017}},
  {\cite{Lurie2018a}} and {\cite{Lurie2018}}. In particular, we will let
  $\tmop{LMod}_R$ denote the $\infty$-category of an $\mathbb{E}_1$-ring $R$,
  let $\tmop{Mod}_R$ denote the symmetric monoidal $\infty$-category of an
  $\mathbb{E}_{\infty}$-ring $R$ and let $\tmop{Alg}_R^{\mathbb{E}_n}$ denote
  the $\infty$-category of $\mathbb{E}_n$-$R$-algebras for an
  $\mathbb{E}_{\infty}$-ring $R$ and a positive integer $n \in \mathbb{N}_{>
  0}$. In particular, we will denote $\tmop{Alg}_R^{\mathbb{E}_{\infty}}$ by
  $\tmop{CAlg}_R$, referred to as the $\infty$-category of commutative
  $R$-algebras. On the other hand, we will denote $\tmop{Mod}_R^{\heartsuit}$
  the $\infty$-category of discrete $R$-modules, and
  $\tmop{CAlg}_R^{\heartsuit}$ the $\infty$-category of discrete commutative
  $R$-algebras.
\end{notation}

\section{Recollection of spherical Witt vectors}\label{sec:ha}

In this section, we will review the definition and some basic properties of
spherical Witt vectors. We quote some definitions and propositions from
{\cite[Section~5.2]{Lurie2018a}}.

\begin{definition}[{\cite[Definition~5.2.1]{Lurie2018a}}]
  \label{def:thickening}Let $A$ be a connective $\mathbb{E}_{\infty}$-ring,
  let $I \subseteq \pi_0 A$ be a finitely generated ideal, and set $A_0 =
  \pi_0 (A) / I$. Suppose that we are given a commutative diagram of
  connective $\mathbb{E}_{\infty}$-rings
  \[ \begin{array}{ccc}
       A & \xrightarrow{f} & B\\
       \longdownarrow & \sigma & \longdownarrow\\
       H A_0 & \xrightarrow{f_0} & H B_0
     \end{array} \]
  where $B_0$ is a discrete commutative ring. We will say that
  {\tmdfn{$\sigma$ exhibits $f$ as an $A$-thickening of $f_0$}} if the
  following conditions are satisfied:
  \begin{enumeratealpha}
    \item The $\mathbb{E}_{\infty}$-ring $B$ is $I$-complete as an $A$-module;
    
    \item The diagram $\sigma$ induces an isomorphism of commutative rings
    $\pi_0 (B) / I \pi_0 (B) \rightarrow B_0$;
    
    \item Let $R$ be any connective $\mathbb{E}_{\infty}$-algebra over $A$
    which is $I$-complete. Then the canonical map
    \[ \tmop{Map}_{\tmop{CAlg}_A} (B, R) \rightarrow
       \tmop{Hom}_{\tmop{CAlg}_{A_0}^{\heartsuit}} (B_0, \pi_0 (R) / I \pi_0
       (R)) \]
    is a homotopy equivalence. In particular, the mapping space
    $\tmop{Map}_{\tmop{CAlg}_A} (B, R)$ is discrete up to homotopy
    equivalence, that is, each connected component is contractible.
  \end{enumeratealpha}
\end{definition}

\begin{remark}[Uniqueness, {\cite[Remark~5.2.2]{Lurie2018a}}]
  \label{rem:thickening-uniq}Let $A$ be a connective
  $\mathbb{E}_{\infty}$-ring, let $I \subseteq \pi_0 A$ be a finitely
  generated ideal, and set $A_0 = \pi_0 (A) / I$. Suppose that we are given a
  homomorphism of commutative rings $f_0 \of A_0 \rightarrow B_0$. It follows
  immediately from the definition that if there exists a diagram $\sigma$:
  \[ \begin{array}{ccc}
       A & \xrightarrow{f} & B\\
       \longdownarrow & \sigma & \longdownarrow\\
       H A_0 & \xrightarrow{f_0} & H B_0
     \end{array} \]
  which exhibits $f$ as an $A$-thickening of $f_0$, then the morphism $f$ (and
  the diagram $\sigma$) is uniquely determined up to equivalence.
\end{remark}

\begin{remark}[{\cite[Remark~5.2.4]{Lurie2018a}}]
  \label{rem:thickeningsqs}Suppose that we are given a commutative diagram of
  commutative $\mathbb{E}_{\infty}$-rings
  \[ \begin{array}{ccc}
       A & \xrightarrow{f} & B\\
       \longdownarrow &  & \longdownarrow\\
       A' & \xrightarrow{f'} & B'\\
       \longdownarrow &  & \longdownarrow\\
       H A_0 & \xrightarrow{f_0} & H B_0
     \end{array} \]
  Assume that $A_0, B_0$ are discrete rings and the left vertical maps induce
  surjective ring morphisms $\pi_0 A \rightarrow \pi_0 A' \rightarrow A_0$
  whose composition has kernel $I \subseteq \pi_0 A$. Suppose that the outer
  rectangle exhibits $f$ as an $A$-thickening of $f_0$ and that the upper
  square exhibits $B'$ as an $I$-completion of $B \otimes_A A'$. Then the
  lower square exhibits $f'$ as an $A'$-thickening of $f_0$.
\end{remark}

\begin{theorem}[{\cite[Theorem~5.2.5]{Lurie2018a}}]
  \label{thm:exist_thickening}Let $A$ be a connective
  $\mathbb{E}_{\infty}$-ring, let $I \subseteq \pi_0 A$ be a finitely
  generated ideal, and set $A_0 = \pi_0 (A) / I$. Suppose that $A_0$ is an
  $\mathbb{F}_p$-algebra such that $H A_0$ is almost perfect as an $A$-module
  and that the Frobenius map $\varphi_{A_0} \of A_0 \rightarrow A_0$ is flat.
  Let $f \of A_0 \rightarrow B_0$ be a morphism of commutative
  $\mathbb{F}_p$-algebras which is relatively perfect, then there exists a
  diagram
  \[ \begin{array}{ccc}
       A & \xrightarrow{f} & B\\
       \longdownarrow & \sigma & \longdownarrow\\
       H A_0 & \xrightarrow{f_0} & H B_0
     \end{array} \]
  which exhibits $f$ as an $A$-thickening of $f_0$. Moreover, $\sigma$ is a
  pushout square.
\end{theorem}

\begin{example}[Classical Witt vectors, {\cite[Example~5.2.6]{Lurie2018a}}]
  \label{ex:witt}In the statement of Theorem~\ref{thm:exist_thickening} take
  $A = H\mathbb{Z}_p$ and $I = p\mathbb{Z}_p$. Then $A_0 = \pi_0 (A) / I$ is
  the finite field $\mathbb{F}_p$ and a map $f_0 \of A_0 \rightarrow B_0$ of
  discrete rings is relative perfect if and only if $B_0$ is a perfect
  $\mathbb{F}_p$-algebra. If this condition is satisfied, then
  Theorem~\ref{thm:exist_thickening} allows us to lift $B_0$ to an
  $\mathbb{E}_{\infty}$-$H\mathbb{Z}_p$-algebra which is complete with respect
  to the ideal $p\mathbb{Z}_p$ and for which the quotient $\pi_0 (B) / p \pi_0
  (B)$ is isomorphic to $B_0$. This $\mathbb{Z}_p$-algebra is in fact the
  Eilenberg-Maclane spectrum of the ring of Witt vectors $W (B_0)$. See also
  {\cite[Section~II.5, Proposition~10]{Serre1979}} for a classical description
  of this universal property.
\end{example}

\begin{example}[Spherical Witt vectors, {\cite[Example~5.2.7]{Lurie2018a}}]
  \label{ex:sphwitt}In the statement of Theorem~\ref{thm:exist_thickening}
  take $A =\mathbb{S}_p^{\wedge}$ and $I = (p)$. Then $A_0 = \pi_0 (A) / I$ is
  the finite field $\mathbb{F}_p$ and a morphism $f_0 \of A_0 \rightarrow B_0$
  is relative perfect if and only if $B_0$ is a perfect
  $\mathbb{F}_p$-algebra. If this condition is satisfied, Theorem
  \ref{thm:exist_thickening} allows us to lift $B_0$ to an
  $\mathbb{E}_{\infty}$-$\mathbb{S}_p^{\wedge}$-algebra which is complete with
  respect to the ideal $(p)$ and the tensor product $H\mathbb{F}_p
  \otimes_{\mathbb{S}_p^{\wedge}} B \simeq \pi_0 (B) / p \pi_0 (B)$ is
  isomorphic to $B_0$. This is the $\mathbb{E}_{\infty}$-ring $W^+ (B_0)$ of
  ``spherical'' Witt vectors.
\end{example}

\begin{proposition}
  \label{prop:sphwitt}$\pi_0 (W^+ (k))$ is isomorphic to $W (k)$, the ring of
  Witt vectors, and $H W (k) \simeq W^+ (k) \otimes_{\mathbb{S}_p^{\wedge}}
  H\mathbb{Z}_p$ for any perfect $\mathbb{F}_p$-algebra $k$.
\end{proposition}

\begin{proof}
  First, we have a commutative diagram
  
  \begin{figure}[H]
    \tmscriptoutput{xypic}{XYpic}{\
    
    \textbackslash xymatrix\{
    
    \textbackslash mathbb S\_p\^{}\textbackslash wedge
    
    \textbackslash ar[r]\textbackslash ar[d]\&
    
    W\^{}+(k)\textbackslash ar[dd]\textbackslash\textbackslash
    
    H\textbackslash mathbb Z\_p\textbackslash
    ar[d]\textbackslash\textbackslash
    
    H\textbackslash mathbb F\_p\textbackslash ar[r]\&Hk
    
    \}}{\raisebox{-0.00244300451383256\height}{\includegraphics[width=3.06021251475797cm,height=3.43660632296996cm]{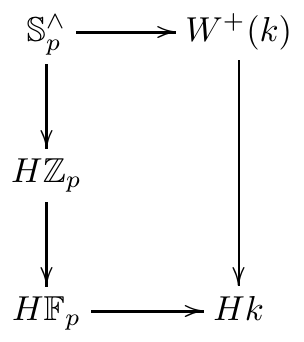}}}
    \caption{}
  \end{figure}
  
  {\noindent}where the outer square is given by
  Theorem~\ref{thm:exist_thickening}. The right vertical map $W^+ (k)
  \rightarrow H k$ factors through the pushout $W^+ (k)
  \otimes_{\mathbb{S}_p^{\wedge}} H\mathbb{Z}_p$ in the category of
  $\mathbb{E}_{\infty}$-rings. Note that $\mathbb{S}_p^{\wedge}$ is a coherent
  ring as in Definition~\ref{def:coherentE1}, and $H\mathbb{Z}_p \simeq H
  \pi_0 (\mathbb{S}_p^{\wedge})$ is an almost perfect
  $\mathbb{S}_p^{\wedge}$-module by Corollary~\ref{cor:coherent_aperf}, which
  implies that $W^+ (k) \otimes_{\mathbb{S}_p^{\wedge}} H\mathbb{Z}_p$ is an
  almost perfect $W^+ (k)$-module by Proposition~\ref{prop:aperf-base-change}.
  By Definition~\ref{def:thickening}, $W^+ (k)$ is a $p$-complete
  $\mathbb{E}_{\infty}$-$\mathbb{S}_p^{\wedge}$-algebra, therefore by
  Proposition~\ref{prop:aperf-complete}, the spectrum $W^+ (k)
  \otimes_{\mathbb{S}_p^{\wedge}} H\mathbb{Z}_p$ is $p$-complete. Now we take
  $A =\mathbb{S}_p^{\wedge}$, $A' = H\mathbb{Z}_p$, $A_0 = H\mathbb{F}_p$, $B
  = W^+ (k)$, $B' = W^+ (k) \otimes_{\mathbb{S}_p^{\wedge}} H\mathbb{Z}_p$ and
  $B_0 = H k$ in Remark~\ref{rem:thickeningsqs}, we deduce that the lower
  square
  \[ \begin{array}{ccc}
       H\mathbb{Z}_p & \longrightarrow & W^+ (k)
       \otimes_{\mathbb{S}_p^{\wedge}} H\mathbb{Z}_p\\
       \longdownarrow &  & \longdownarrow\\
       H\mathbb{F}_p & \longrightarrow & H k
     \end{array} \]
  constitutes a commutative diagram of thickening as in
  Definition~\ref{def:thickening}. Then it follows from
  Remark~\ref{rem:thickening-uniq} and Example~\ref{ex:witt} that $W^+ (k)
  \otimes_{\mathbb{S}_p^{\wedge}} H\mathbb{Z}_p$ is equivalent to $H W (k)$ as
  $\mathbb{E}_{\infty}$-$H\mathbb{Z}_p$-algebras, which implies that $W (k)
  \cong \pi_0 (H W (k)) \cong \tmop{Tor}_0^{\pi_0 (\mathbb{S}_p^{\wedge})}
  (\pi_0 (W^+ (k)), \pi_0 (H\mathbb{Z}_p)) \cong \tmop{Tor}_0^{\mathbb{Z}_p}
  (\pi_0 (W^+ (k)), \mathbb{Z}_p) \cong \pi_0 (W^+ (k))$.
\end{proof}

\begin{proposition}[Recognition of Thickenings,
{\cite[Proposition~5.2.9]{Lurie2018a}}]
  Let $A$ be a connective $\mathbb{E}_{\infty}$-ring, let $I \subseteq \pi_0
  A$ be a finitely generated ideal, and set $A_0 = \pi_0 (A) / I$. Suppose
  that $A_0$ is an $\mathbb{F}_p$-algebra which is almost perfect as an
  $A$-module and that the Frobenius map $\varphi_{A_0} \of A_0 \rightarrow
  A_0$ is flat. Suppose we are given a commutative diagram of connective
  $\mathbb{E}_{\infty}$-rings $\sigma$:
  \[ \begin{array}{ccc}
       A & \xrightarrow{f} & B\\
       \longdownarrow & \sigma & \longdownarrow\\
       A_0 & \xrightarrow{f_0} & B_0
     \end{array} \]
  where $f_0$ is a relative perfect morphism of commutative
  $\mathbb{F}_p$-algebras. Then $\sigma$ exhibits $f$ as an $A$-thickening of
  $f_0$ if and only if the following conditions are satisfied:
  \begin{enumerateroman}
    \item The $\mathbb{E}_{\infty}$-ring $B$ is $I$-complete as an $A$-module;
    
    \item The diagram $\sigma$ is a pushout square.
  \end{enumerateroman}
\end{proposition}

\section{Perfect rings being Thom spectra}

We first admit a (superficially) slightly stronger Hopkins-Mahowald's theorem
for sake of convenience. Given a perfect $\mathbb{F}_p$-algebra $k$ and an
invertible element $u \in \tmop{GL}_1 (W (k))$, as a special case of
Remark~\ref{cons:f_R}, we have a map $f_{k, pu} \of \Omega^2 S^3 \rightarrow
\tmop{BGL}_1 (W^+ (k))$.

\begin{theorem}[Hopkins-Mahowald for $k$]
  \label{thm:perf-thom}The Eilenberg-Maclane spectrum $H k$ is the
  $\mathbb{E}_2$-Thom spectrum associated to the map $f_{k, pu}$.
\end{theorem}

For technical reasons, we start with the special case that $u \in \tmop{GL}_1
(\mathbb{Z}_p) \subseteq \tmop{GL}_1 (W (k))$. In this case, it is a direct
consequence of that for $\mathbb{F}_p .$

\begin{lemma}
  \label{lem:perf-thom-special}Theorem~\ref{thm:perf-thom} is true when $u \in
  \tmop{GL}_1 (\mathbb{Z}_p) \subseteq \tmop{GL}_1 (W (k))$.
\end{lemma}

\begin{proof}
  We note that the image of the multiplication map $m_{1 - pu} \of
  \mathbb{S}_p^{\wedge} \rightarrow \mathbb{S}_p^{\wedge}$ given by $1 - pu
  \in \pi_0 (\mathbb{S}_p^{\wedge}) \cong \mathbb{Z}_p$ under the canonical
  (symmetric monoidal) functor $W^+ (k) \otimes_{\mathbb{S}_p^{\wedge}} - \of
  \tmop{Mod}_{\mathbb{S}_p^{\wedge}} \rightarrow \tmop{Mod}_{W^+ (k)}$ is
  still a multiplication map $m_{1 - pu} \of W^+ (k) \rightarrow W^+ (k)$
  given by $1 - pu \in \pi_0 (W^+ (k)) \cong W (k)$, and therefore the map
  $f_{k, pu}$ coincides with the composition map
  \[ \Omega^2 S^3 \longrightarrowlim^{f_{\mathbb{F}_p, pu}} \tmop{BGL}_1
     (\mathbb{S}_p^{\exterior}) \xrightarrow{W^+ (k)
     \otimes_{\mathbb{S}_p^{\wedge}} -} \tmop{BGL}_1 (W^+ (k)) \]
  Since $M f_{\mathbb{F}_p, pu} \simeq H\mathbb{F}_p$ as $\mathbb{E}_2$-ring
  spectra by Theorem~\ref{thm:hopkins-mahowald},
  \[ M f_{k, pu} \simeq W^+ (k) \otimes_{\mathbb{S}_p^{\wedge}} M
     f_{\mathbb{F}_p, pu} \simeq W^+ (k) \otimes_{\mathbb{S}_p^{\wedge}}
     H\mathbb{F}_p \simeq H k \]
  as $\mathbb{E}_2$-ring spectra, where the first equivalence follows from the
  fact that the functor $W^+ (k) \otimes_{\mathbb{S}_p^{\wedge}} -$ is a left
  adjoint therefore commutes with colimits and the last equivalence is given
  by the last claim in Theorem~\ref{thm:exist_thickening}.
\end{proof}

In order to prove Theorem~\ref{thm:perf-thom}, it suffices to show that $M
f_{k, pu} \simeq M f_{k, p}$ holds for all $u \in \tmop{GL}_1 (W (k))$,
therefore $M f_{k, pu} \simeq M f_{k, p} \simeq H k$ by
Lemma~\ref{lem:perf-thom-special}. We will base the proof on a universal
property of Thom spectra which we will not use elsewhere, and the author looks
forward to an alternative proof which does not depend on this universal
property.

\begin{lemma}[Proposition 4.9 in {\cite{Antolin-Camarena2014}} along with the
discussions after Lemma 4.6]
  \label{lem:thom-universal}The $\mathbb{E}_2$-Thom spectrum $M f_{k, pu}$
  satisfies the following universal property: For all $\mathbb{E}_2$-$W^+
  (k)$-algebras $A$, the mapping space $\tmop{Map}_{\tmop{Alg}_{W^+
  (k)}^{\mathbb{E}_2}} (M f_{k, pu}, A)$ could be naturally identified with
  the space of null-homotopies of the composite map $W^+ (k)
  \xrightarrow{m_{pu}} W^+ (k) \xrightarrow{\eta} A$ in the category of $W^+
  (k)$-modules where $\eta \of W^+ (k) \rightarrow A$ is the canonical map
  given by the $\mathbb{E}_2$-$W^+ (k)$-algebra structure, and $m_{pu} \of W^+
  (k) \rightarrow W^+ (k)$ is the multiplication map given by $pu \in W (k) =
  \pi_0 (W^+ (k))$.
\end{lemma}

\begin{proof*}{Proof of Theorem~\ref{thm:perf-thom}}
  Note that the multiplication map $m_u \of W^+ (k) \rightarrow W^+ (k)$ is an
  equivalence of $W^+ (k)$-modules since $u \in W (k) = \pi_0 (W^+ (k))$ is
  invertible. Hence by Lemma~\ref{lem:thom-universal}, the map $m_u$ induces
  an equivalence of spaces $\tmop{Map}_{\tmop{Alg}_{W^+ (k)}^{\mathbb{E}_2}}
  (M f_{k, p}, A) \rightarrow \tmop{Map}_{\tmop{Alg}_{W^+ (k)}^{\mathbb{E}_2}}
  (M f_{k, pu}, A)$ which is natural in $A$. By the Yoneda lemma, we deduce
  that $M f_{k, pu} \simeq M f_{k, p}$ as $\mathbb{E}_2$-$W^+ (k)$-algebras.
\end{proof*}

\section{Recollection of perfectoid rings}\label{sec:perfd}

In this section, we will review basic definitions and properties of perfectoid
rings.

\subsection{Basic definitions and properties}

\begin{definition}
  Let $A$ be a ring and $I \subseteq A$ be an ideal. Then the ring $A$ is
  called {\tmdfn{$I$-adically complete}} if the canonical map from $A$ to the
  (inverse) limit of the tower
  \[ \cdots \rightarrow A / I^n \rightarrow \cdots \rightarrow A / I^2
     \rightarrow A / I \]
  is an isomorphism. The ring $A$ is called $I$-adically separated if the
  intersection $\bigcap_{n = 0}^{\infty} I^n = 0$.
\end{definition}

\begin{warning}
  In the literature, sometimes authors call a ring $A$ is $I$-adically
  complete when the canonical map $A \rightarrow \lim_{n \in (\mathbb{N}, >)}
  (A / I^n)$ is only supposed to be surjective, and our $I$-adic completeness
  is equivalent to their $I$-adic completeness plus $I$-adic separateness.
\end{warning}

\begin{definition}
  Let $A$ be an $\mathbb{F}_p$-algebra. The {\tmdfn{direct limit perfection}}
  $A_{\tmop{perf}}$ of $A$ is the direct limit of the telescope $A
  \xrightarrow{\varphi} A \xrightarrow{\varphi} A \xrightarrow{\varphi}
  \cdots$.
\end{definition}

\begin{definition}
  An $\mathbb{F}_p$-algebra $A$ is called {\tmdfn{semiperfect}} if the
  Frobenius map $\varphi \of A \rightarrow A$ is surjective.
\end{definition}

\begin{remark}
  For a semiperfect $\mathbb{F}_p$-algebra $A$, the direct limit perfection
  $A_{\tmop{perf}}$ coincides with $A_{\tmop{red}} = A / \sqrt{0}$, by
  checking that $A_{\tmop{red}}$ satisfies the universal property of
  $A_{\tmop{perf}}$.
\end{remark}

\begin{remark}
  The canonical map $R \rightarrow R_{\tmop{perf}}$ is initial among all
  $\mathbb{F}_p$-algebra morphisms $R \rightarrow S$ such that $S$ is a
  perfect $\mathbb{F}_p$-algebra. This follows directly from the universal
  property of direct limits in the definition of direct limit perfections.
\end{remark}

\begin{definition}
  Let $R$ be a commutative ring which is $p$-adically complete. The
  {\tmdfn{tilt}} of $R$, denoted by $R^{\flat}$, is a perfect
  $\mathbb{F}_p$-algebra defined by the limit of the tower
  \[ \cdots \xrightarrow{\varphi} R / p \xrightarrow{\varphi} R / p
     \xrightarrow{\varphi} R / p \]
  where $\varphi \of R / p \rightarrow R / p$ is the Frobenius map. In
  particular, if $R$ is an $\mathbb{F}_p$-algebra, then $R^{\flat}$ is the
  inverse limit perfection of $R$, and if furthermore $R$ is semiperfect, then
  the canonical map $R^{\flat} \rightarrow R$ is a surjection.
\end{definition}

We need the following classical proposition to define the Fontaine's
pro-infinitesimal thickening map. We omit the proof which is routine. One can
find a proof in, say, {\cite[Section~1.3]{Hesselholt2019}}.

\begin{proposition}
  Let $R$ be a $p$-adically complete commutative ring. Then there exists a
  multiplicative map (that is to say, a morphism of multiplicative monoids)
  $R^{\flat} \xrightarrow{(-)^{\sharp}} R$ that sends $a = (x_n)_{n \in
  \mathbb{N}} \in R^{\flat}$ to $a^{\sharp} \assign \lim_{n \rightarrow
  \infty} y_n^{p^n}$ where $(x_n)_{n \in \mathbb{N}}$ satisfies $\varphi (x_{n
  + 1}) = x_n$ for all $n \in \mathbb{N}$, and $(y_n)_{n \in \mathbb{N}} \in
  R^{\mathbb{N}}$ is any sequence such that for each $n \in \mathbb{N}$, $y_n$
  is a lift of $x_n \in R / p$ in $R$. We note that $a^{\sharp}$ does not
  depend on choice of $(y_n)_{n \in \mathbb{N}}$.
\end{proposition}

\begin{definition}
  {\tmdfn{Fontaine's map}} $\theta \of W (R^{\flat}) \rightarrow R$ is given
  by $\theta \left( \sum_{i = 0}^{\infty} [a_i] p^i \right) = \sum_{i =
  0}^{\infty} a_i^{\sharp} p^i$, where $[-] \of R^{\flat} \rightarrow W
  (R^{\flat})$ is the Teichmüller representative.
\end{definition}

\begin{definition}[{\cite[Definition 3.5]{Bhatt2016}}]
  \label{def:perfd}A commutative ring $R$ is {\tmdfn{perfectoid}} if there
  exists $\pi \in R$ such that $p \in \pi^p R$, such that the ring $R$ is
  $(\pi)$-adically complete, such that the $\mathbb{F}_p$-algebra $R / p$ is
  semiperfect, and such that the kernel of $\theta \of W (R^{\flat})
  \rightarrow R$ is a principal ideal.
\end{definition}

\begin{definition}
  Let $R$ be a perfectoid ring. The {\tmdfn{special fiber}}, denoted by
  $\kappa$, is the direct limit perfection of $R / p$, that is to say $\kappa
  \assign (R / p)_{\tmop{perf}} = R / \sqrt{pR}$ since $R / p$ is semiperfect.
\end{definition}

\begin{notation}
  Let $R$ be a perfectoid ring. We denote by $\xi$ a generator of Fontaine's
  map $\theta \of W (R^{\flat}) \rightarrow R$.
\end{notation}

\begin{proposition}[{\cite[Lemma~3.13]{Bhatt2016}}]
  \label{prop:perfdPOfiber}Let $R$ be a perfectoid ring. Then the commutative
  diagram
  \[ \begin{array}{ccc}
       W (R^{\flat}) & \xrightarrow{\theta} & R\\
       \longdownarrow &  & \longdownarrow\\
       W (\kappa) & \xrightarrow{\tmop{mod} p} & \kappa
     \end{array} \]
  is a $\tmop{Tor}$-independent pushout square.
\end{proposition}

\begin{corollary}
  \label{cor:image-xi}Let $R$ be a perfectoid ring. For any generator $\xi \in
  \ker \theta$, there exists an invertible element $u \in \tmop{GL}_1 (W
  (\kappa))$ such that the image of $\xi \in W (R^{\flat})$ in $W (\kappa)$ is
  $pu$.
\end{corollary}

\begin{proof}
  By Proposition~\ref{prop:perfdPOfiber}, the image of $u \in W (R^{\flat})$
  in $W (\kappa)$ is a generator of the ideal $pW (\kappa)$. Since $p \in W
  (\kappa)$ is not a zero divisor, we deduce the result that we need. 
\end{proof}

\begin{proposition}
  \label{prop:ker-fiber}Let $R$ be a perfectoid ring. Then the kernel of the
  composition $R^{\flat} \rightarrow R / p \rightarrow \kappa$ is $\sqrt{\xi
  R^{\flat}}$.
\end{proposition}

\begin{proof}
  The kernel of the composition $W (R^{\flat}) \rightarrow R / p \rightarrow
  \kappa$ is $\sqrt{pW (R^{\flat}) + \xi W (R^{\flat})}$ whose image under the
  canonical map $W (R^{\flat}) \rightarrow R^{\flat}$ is $\sqrt{\xi
  R^{\flat}}$.
\end{proof}

\subsection{Universal properties of Fontaine's map (and a spherical analogue)}

The results of this subsection will not be used later. However, we find it
better to understand that Fontaine's map $\theta \of W (R^{\flat}) \rightarrow
R$ and its spherical analogue $W^+ (R^{\flat}) \rightarrow \tau_{\leq 0} (W^+
(R^{\flat})) \simeq H W (R^{\flat}) \xrightarrow{H \theta} H R$ satisfy a
universal property, which is related to the thickening defined in
Definition~\ref{def:thickening}. Roughly speaking, they are mixed
characteristic ``absolute'' versions of thickenings in
Definition~\ref{def:thickening}. The following proposition is essentially due
to Fontaine (see {\cite{Fontaine1994}}, Theorem~1.2.1), rephrased in the
modern language:

\begin{proposition}[{\cite[Proposition~1.3.4]{Hesselholt2019}}]
  \label{prop:Fontaine-inf-thickening}Let $R$ be a perfectoid ring. Then
  Fontaine's map $\theta \of W (R^{\flat}) \rightarrow R$ is initial among
  surjections $\theta_D \of D \rightarrow R$ of rings such that the ring $D$
  is both $p$-adically complete and $\ker \theta_D$-adically complete.
\end{proposition}

We will sketch the proof in {\cite{Hesselholt2019}} for the universal
property, that is, assume that the $p$-adic completeness and the $\xi$-adic
completeness of $W (R^{\flat})$ is already given, we show that Fontaine's map
$\theta \of W (R^{\flat}) \rightarrow R$ is initial as claimed.

\begin{proof}
  Let $\theta_D \of D \rightarrow R$ be a map of rings such that $D$ is both
  $p$-adically complete and $\ker \theta_D$-adically complete. We need to show
  that $\theta_D$ factors uniquely through $\theta \of W (R^{\flat})
  \rightarrow R$. In view of Example~\ref{ex:witt} and
  Definition~\ref{def:thickening}, we have a bijection
  \[ \tmop{Hom}_{\tmop{CAlg}_{\mathbb{Z}_p}^{\heartsuit}} (W (R^{\flat}), D)
     \xrightarrow{\cong} \tmop{Hom}_{\tmop{CAlg}_{\mathbb{F}_p}^{\heartsuit}}
     (R^{\flat}, D / p) \]
  (here everything is discrete therefore classical, but in order to avoid
  conflicts of notations with other parts of the article, we retain the
  cumbersome notations $\tmop{CAlg}_{\mathbb{Z}_p}^{\heartsuit}$ and
  $\tmop{CAlg}_{\mathbb{F}_p}^{\heartsuit}$) which is given as follows: for
  any map $W (R^{\flat}) \rightarrow D$ of discrete $\mathbb{Z}_p$-algebras,
  we compose it with the canonical map $D \rightarrow D / p$ to get the map $W
  (R^{\flat}) \rightarrow D / p$, which factors uniquely through $W
  (R^{\flat}) \rightarrow W (R^{\flat}) / p \cong R^{\flat}$ therefore gives
  rise to a map $R^{\flat} \rightarrow D / p$. Note that $\theta_R =
  \tmop{id}_R \of R \rightarrow R$ serves as a special choice of $\theta_D$
  since the perfectoid ring $R$ is $p$-adically complete by
  Definition~\ref{def:perfd} and tautologically $\ker (\theta_R) =
  (0)$-adically complete. That is to say, we also have a bijection
  \[ \tmop{Hom}_{\tmop{CAlg}_{\mathbb{Z}_p}^{\heartsuit}} (W (R^{\flat}), R)
     \xrightarrow{\cong} \tmop{Hom}_{\tmop{CAlg}_{\mathbb{F}_p}^{\heartsuit}}
     (R^{\flat}, R / p) \]
  The map $\theta_D \of D \rightarrow R$ gives rise to a commutative diagram
  \[ \begin{array}{ccc}
       \tmop{Hom}_{\tmop{CAlg}_{\mathbb{Z}_p}^{\heartsuit}} (W (R^{\flat}), D)
       & \xrightarrow{\cong} &
       \tmop{Hom}_{\tmop{CAlg}_{\mathbb{F}_p}^{\heartsuit}} (R^{\flat}, D /
       p)\\
       \longdownarrow &  & \longdownarrow\\
       \tmop{Hom}_{\tmop{CAlg}_{\mathbb{Z}_p}^{\heartsuit}} (W (R^{\flat}), R)
       & \xrightarrow{\cong} &
       \tmop{Hom}_{\tmop{CAlg}_{\mathbb{F}_p}^{\heartsuit}} (R^{\flat}, R / p)
     \end{array} \]
  So in order to show that the map $\theta_D \of D \rightarrow R$ factors
  through the canonical map $\theta$, or equivalently put, the preimage of the
  element $\theta \in \tmop{Hom}_{\tmop{CAlg}_{\mathbb{Z}_p}^{\heartsuit}} (W
  (R^{\flat}), R)$ under the induced map
  $\tmop{Hom}_{\tmop{CAlg}_{\mathbb{Z}_p}^{\heartsuit}} (W (R^{\flat}), D)
  \rightarrow \tmop{Hom}_{\tmop{CAlg}_{\mathbb{Z}_p}^{\heartsuit}} (W
  (R^{\flat}), R)$ is a singleton, it suffices to show that the preimage of
  the element $(R^{\flat} \rightarrow R / p) \in
  \tmop{Hom}_{\tmop{CAlg}_{\mathbb{F}_p}^{\heartsuit}} (R^{\flat}, R / p)$
  under the map $\tmop{Hom}_{\tmop{CAlg}_{\mathbb{F}_p}^{\heartsuit}}
  (R^{\flat}, D / p) \rightarrow
  \tmop{Hom}_{\tmop{CAlg}_{\mathbb{F}_p}^{\heartsuit}} (R^{\flat}, R / p)$ is
  a singleton, or equivalently put, the canonical map $\sigma \of R^{\flat}
  \rightarrow R / p$ lifts uniquely through the map $\sigma_D \of D / p
  \rightarrow R / p$ induced by the map $\theta_D \of D \rightarrow R$. Note
  that the surjectivity of $\theta_D$ implies that of $\sigma_D$. Since the
  ring $D$ is $(p, \ker \theta_D)$-adically complete, the ring $D / p$ is
  $\ker \sigma_D$-adically complete.
  
  We can conclude the existence and the uniqueness of lift of the map $\sigma
  \of R^{\flat} \rightarrow R / p$ along the surjection $\sigma_D \of D / p
  \rightarrow R / p$ simply by the fact that the $\mathbb{F}_p$-algebra
  $R^{\flat}$ is perfect and thus the cotangent complex $\mathbb{L}_{R^{\flat}
  /\mathbb{F}_p}$ is contractible, which implies the existence and the
  uniqueness of such lift.
  
  However, we prefer to give a direct argument: We set the
  $\mathbb{F}_p$-algebra $A \assign R^{\flat}$ to stress that we only depend
  on the fact that $A$ is a perfect $\mathbb{F}_p$-algebra, but not on the
  properties of the map $\sigma \of A \rightarrow R / p$. Denote by $\varphi_B
  \of B \rightarrow B, x \mapsto x^p$ the Frobenius map on any
  $\mathbb{F}_p$-algebra $A$. Then the Frobenius map $\varphi_A$ is an
  isomorphism by assumption.
  
  For each $a \in A$, we choose a sequence $(b_n)_{n = 0}^{\infty} \in (D /
  p)^{\mathbb{N}}$ such that for each $n \in \mathbb{N}$, we have $\sigma_D
  (b_n) = \sigma (\varphi_A^{- n} (a))$.
  
  Note that the sequence $(\varphi_{D / p}^n (b_n))_{n = 0}^{\infty}$
  converges $\ker \sigma_D$-adically: $\sigma_D (b_n - b_{n + 1}^p) = \sigma_D
  (b_n) - \sigma_D (b_{n + 1})^p = \sigma (\varphi_A^{- n} (a)) - \sigma
  (\varphi_A^{- (n + 1)} (a))^p = \sigma (\varphi_A^{- n} (a)) - \sigma
  (\varphi_A (\varphi_A^{- n} (a))) = 0$ and therefore $\varphi_{D / p}^n
  (b_n) - \varphi_{D / p}^{n + 1} (b_{n + 1}) = \varphi_{D / p}^n (b_n - b_{n
  + 1}^p) \in \varphi_{D / p}^n (\ker \sigma_D) \subseteq (\ker
  \sigma_D)^{p^n}$. Let $b \assign \lim_{n \rightarrow \infty} \varphi_{D /
  p}^n (b_n)$.
  
  We first note that $\sigma_D (b) = \sigma (a)$, since $\sigma (\varphi_{D /
  p}^n (b_n)) = \sigma (b_n^{p^n}) = \sigma (b_n)^{p^n} = \sigma (\varphi_A^{-
  n} (a))^{p^n} = \sigma (\varphi_A^n (\varphi_A^{- n} (a))) = \sigma (a)$ for
  all $n \in \mathbb{N}$.
  
  Now the value $b \in D / p$ does not depend on the choice of $(b_n)$, since
  for any other choice $(c_n)$, we have $c_n - b_n \in \ker \sigma_D$, thus
  $\varphi_{D / p}^n (b_n) - \varphi_{D / p}^n (c_n) = \varphi_{D / p}^n (b_n
  - c_n) \in \varphi_{D / p}^n (\ker \sigma_D) \subseteq (\ker
  \sigma_D)^{p^n}$ which implies that $\lim_{n \rightarrow \infty} \varphi_{D
  / p}^n (c_n) = b$.
  
  Combining the preceding discussions, we have shown that for each $a \in A$,
  we can associate a $b \in D / p$ such that $\sigma_D (b) = \sigma (a)$. It
  is routine to check that $a \mapsto b$ defines a map $A \rightarrow D / p$
  of rings which serves as a lift of $\sigma : A \rightarrow R / p$.
  Furthermore, the uniqueness essentially follows from the above argument that
  the value $b \in D / p$ does not depend on the choice of $(b_n)$.
\end{proof}

\begin{remark}
  \label{rem:Fontaine-deriv-cpl}In fact, we can weaken our assumption on $D$
  to be derived $p$-complete and that the map $D \rightarrow R$ is Adams
  complete (due to {\cite{Carlsson2008}} while the terminology is coined in
  {\cite{Bhatt2012}}) by using some basic facts about Adams complete
  surjective maps of animated rings.
\end{remark}

Now we give a spherical version of Fontaine's universal property:

\begin{proposition}
  \label{prop:sph-Fontaine-inf-thickening}Let $R$ be a perfectoid ring. We
  compose Fontaine's map $\theta \of W (R^{\flat}) \rightarrow R$ with the 0th
  Postnikov section $W^+ (R^{\flat}) \rightarrow \tau_{\leq 0} (W^+
  (R^{\flat})) = H W (R^{\flat})$, obtaining the map $\eta \of W^+ (R^{\flat})
  \rightarrow H R$. Then we have
  \begin{enumeratenumeric}
    \item The $\mathbb{E}_{\infty}$-$\mathbb{S}_p^{\wedge}$-algebra $W^+
    (R^{\flat})$ of spherical Witt vectors is $(p, \ker \theta)$-complete.
    Furthermore, the discrete ring $\pi_0 (W^+ (R^{\flat})) / p$ is $\ker
    \theta / (p \pi_0 (W^+ (R^{\flat})) + \ker \theta)$-adically separated.
    
    \item The map $\eta \of W^+ (R^{\flat}) \rightarrow H R$ is initial among
    all maps $\eta_D \of D \rightarrow H R$ surjective on $\pi_0$ where $D$ is
    an $\mathbb{E}_{\infty}$-$\mathbb{S}_p^{\wedge}$-algebra such that $D$ is
    $(p, \ker \eta_D)$-complete and the discrete ring $\pi_0 (D) / p$ is $\ker
    \theta_D / (p \pi_0 (D) + \ker \theta_D)$-adically separated, where we
    denote the map $\pi_0 (\eta_D) \of \pi_0 (D) \rightarrow R$ by $\theta_D$.
  \end{enumeratenumeric}
\end{proposition}

\begin{remark}
  In Proposition~\ref{prop:sph-Fontaine-inf-thickening}, the technical
  conditions imposed on the
  $\mathbb{E}_{\infty}$-$\mathbb{S}_p^{\wedge}$-algebra $D$ are somewhat
  complicated. However, we can restrict to the full subcategory of $\eta_D$
  such that $\pi_0 (D)$ is $(p, \ker \eta_D)$-adically complete, where $\eta
  \of W^+ (R^{\flat}) \rightarrow H R$ lives (see the proof of
  Proposition~\ref{prop:sph-Fontaine-inf-thickening}) and hence $\eta$ is
  still an initial object in this full subcategory.
\end{remark}

\begin{remark}
  Using Remark~\ref{rem:Fontaine-deriv-cpl}, we can drop the adic completeness
  of $\pi_0 (D) / p$ in Proposition~\ref{prop:sph-Fontaine-inf-thickening}.
\end{remark}

Now we want to establish some computational results about homotopy groups of
the ring $W^+ (k)$ of spherical Witt vectors of a perfect
$\mathbb{F}_p$-algebra $k$. First, we need the following proposition, which
follows from Serre's computations of homotopy groups of spheres:

\begin{proposition}
  The sphere spectrum $\mathbb{S}$ is connective, $\pi_0 (\mathbb{S})
  =\mathbb{Z}$, and for all $n \in \mathbb{N}_{> 0}$, the $n$th (stable)
  homotopy group $\pi_n (\mathbb{S})$ is finite.
\end{proposition}

Thus for each $n \in \mathbb{N}$, the homotopy group $\pi_n (\mathbb{S})$ has
bounded $p$-torsion. Combined with Milnor sequence of homotopy groups, we have

\begin{corollary}
  \label{cor:homotopy-grps-sph-spec}The $p$-adic sphere spectrum
  $\mathbb{S}_p^{\wedge}$ is connective, $\pi_0 (\mathbb{S}_p^{\wedge})
  =\mathbb{Z}_p$ and for all $n \in \mathbb{N}_{> 0}$, the $n$th (stable)
  homotopy group $\pi_n (\mathbb{S}_p^{\wedge})$ is a finite direct sum of
  finite abelian groups of form $\mathbb{Z}/ p^r \cong \mathbb{Z}_p / p^r$ for
  some positive integer $r \in \mathbb{N}_{> 0}$.
\end{corollary}

We need a result announced in {\cite[Example~5.2.7]{Lurie2018a}} the argument
of which we learn from Matthew Morrow:

\begin{proposition}
  Let $k$ be a perfect $\mathbb{F}_p$-algebra. Then the ring of spherical Witt
  vectors $W^+ (k)$ is a flat $\mathbb{S}_p^{\wedge}$-module.
\end{proposition}

\begin{proof}
  First, by Proposition~\ref{prop:sphwitt}, $\pi_0 (W^+ (k)) \cong W (k)$
  which is a torsion-free $\mathbb{Z}_p$-module. Since $\mathbb{Z}_p$ is a
  valuation ring, we deduce that $W (k)$ is a flat $\mathbb{Z}_p$-module (see
  {\cite[\href{https://stacks.math.columbia.edu/tag/0539}{Tag
  0539}]{stacks-project}}). Now we consider the Postnikov tower $(\tau_{\geq
  n} \mathbb{S}_p^{\wedge})_{n \in \mathbb{N}}$ of the $p$-adic sphere
  spectrum $\mathbb{S}_p^{\wedge}$, which induces a tower $X_n \assign
  (\tau_{\geq n} \mathbb{S}_p^{\wedge}) \otimes_{\mathbb{S}_p^{\wedge}} W^+
  (k)$. Note that $X_n / X_{n - 1} \cong \Sigma^n H \pi_n
  (\mathbb{S}_p^{\wedge}) \otimes_{\mathbb{S}_p^{\wedge}} W^+ (k)$. We have
  shown in Corollary~\ref{cor:homotopy-grps-sph-spec} that $\pi_n
  (\mathbb{S}_p^{\wedge})$ is a direct sum of finite abelian groups of form
  $\mathbb{Z}_p / p^r$, which allows us to realize $H \pi_n
  (\mathbb{S}_p^{\wedge})$ as a direct sum of spectra of form $\tmop{cofib}
  \left( H\mathbb{Z}_p \xrightarrow{p^r} H\mathbb{Z}_p \right)$. Note that the
  smash product $- \otimes_{\mathbb{S}_p^{\wedge}} W^+ (k)$ commutes with
  taking cofibers, we deduce that $H \pi_n (\mathbb{S}_p^{\wedge})
  \otimes_{\mathbb{S}_p^{\wedge}} W^+ (k) \cong H \tmop{Tor}_0^{\mathbb{Z}_p}
  (\pi_n (\mathbb{S}_p^{\wedge}), W (k))$. Thus the tower $(X_n)_{n \in
  \mathbb{N}}$ constitutes the Postnikov tower of the spectrum $W^+ (k)$,
  therefore $\pi_n (W^+ (k)) \cong \tmop{Tor}_0^{\pi_0
  (\mathbb{S}_p^{\wedge})} (\pi_n (\mathbb{S}_p^{\wedge}), W (k))$.
\end{proof}

\begin{corollary}
  \label{cor:homot-grps-sph-witt}Let $k$ be a perfect $\mathbb{F}_p$-algebra.
  Then the ring of spherical Witt vectors $W^+ (k)$ is connective, $\pi_0 (W^+
  (k)) = W (k)$, and for all $n \in \mathbb{N}_{> 0}$, the $n$th (stable)
  homotopy group $\pi_n (W^+ (k))$ is a finite direct sum of $W (k)$-modules
  of form $W (k) / p^r$.
\end{corollary}

We are now ready to prove Proposition~\ref{prop:sph-Fontaine-inf-thickening}:

\begin{proof*}{Proof of Proposition~\ref{prop:sph-Fontaine-inf-thickening}}
  We check two statements one by one:
  \begin{enumeratenumeric}
    \item Proposition~\ref{prop:Fontaine-inf-thickening} tells us that the
    discrete ring $\pi_0 (W^+ (R^{\flat})) \cong W (R^{\flat})$ is $(p, \ker
    \theta)$-adically complete, therefore by Proposition~\ref{prop:I-adic}, it
    is $(p, \ker \theta)$-complete. Furthermore, we deduce from $(p, \ker
    \theta)$-adic completeness that $\pi_0 (W^+ (R^{\flat}))$ is $\ker \theta
    / (p \pi_0 (W^+ (R^{\flat})) + \ker \theta)$-adically separated. In view
    of Theorem~\ref{thm:cplhomot}, it remains to show that for each $n \in
    \mathbb{N}_{> 0}$, the homotopy group $\pi_n (W^+ (R^{\flat}))$ is
    (derived) $(p, \ker \theta)$-complete as a discrete $W
    (R^{\flat})$-module. However, by Corollary~\ref{cor:homot-grps-sph-witt},
    we have realized $\pi_n (W^+ (R^{\flat}))$ as a direct sum of cofibers of
    $(p, \ker \theta)$-complete modules, therefore it is $(p, \ker
    \theta)$-complete.
    
    \item This part is parallel to the proof of
    Proposition~\ref{prop:Fontaine-inf-thickening}. We start with the
    following commutative diagram induced by the map $\eta_D \of D \rightarrow
    H R$:
    \[ \begin{array}{ccc}
         \tmop{Map}_{\tmop{CAlg}_{\mathbb{S}_p^{\wedge}}} (W^+ (R^{\flat}), D)
         & \xrightarrow{\simeq} &
         \tmop{Hom}_{\tmop{CAlg}_{\mathbb{F}_p}^{\heartsuit}} (R^{\flat},
         \pi_0 (D) / p)\\
         \longdownarrow &  & \longdownarrow\\
         \tmop{Map}_{\tmop{CAlg}_{\mathbb{S}_p^{\wedge}}} (W^+ (R^{\flat}), H
         R) & \xrightarrow{\simeq} &
         \tmop{Hom}_{\tmop{CAlg}_{\mathbb{F}_p}^{\heartsuit}} (R^{\flat}, R /
         p)
       \end{array} \]
    as in the proof of Proposition~\ref{prop:Fontaine-inf-thickening}. It
    follows from Definition~\ref{def:thickening} and Example~\ref{ex:sphwitt}
    that the horizontal maps are homotopy equivalences, which implies that the
    connected components of each space on the left are all contractible. We
    pick the connected component of
    $\tmop{Map}_{\tmop{CAlg}_{\mathbb{S}_p^{\wedge}}} (W^+ (R^{\flat}), H R)$
    corresponds to the map $\eta \of W^+ (R^{\flat}) \rightarrow H R$. In
    order to show that $\eta$ is an initial object, it suffices to show that
    there exists one and only one connected component in
    $\tmop{Map}_{\tmop{CAlg}_{\mathbb{S}_p^{\wedge}}} (W^+ (R^{\flat}), D)$
    which maps to the connect component corresponding to $\eta$. Note that the
    image of $\eta$ in $\tmop{Hom}_{\tmop{CAlg}_{\mathbb{F}_p}^{\heartsuit}}
    (R^{\flat}, R / p)$ along the bottom horizontal map coincides with $\sigma
    \in \tmop{Hom}_{\tmop{CAlg}_{\mathbb{F}_p}^{\heartsuit}} (R^{\flat}, R /
    p)$ defined in the proof of
    Proposition~\ref{prop:Fontaine-inf-thickening}. In view of the commutative
    diagram, it remains to show that the preimage of $\sigma \in
    \tmop{Hom}_{\tmop{CAlg}_{\mathbb{F}_p}^{\heartsuit}} (R^{\flat}, R / p)$
    in $\tmop{Hom}_{\tmop{CAlg}_{\mathbb{F}_p}^{\heartsuit}} (R^{\flat}, \pi_0
    (D) / p)$. The rest of the proof is identical to that of
    Proposition~\ref{prop:Fontaine-inf-thickening}.
  \end{enumeratenumeric}
\end{proof*}

\section{Proof of the main theorem}\label{sec:main}

Fix a perfectoid ring $R$ and a generator $\xi$ of Fontaine's map $\theta \of
W (R^{\flat}) \rightarrow R$, the goal of this section is to prove
Theorem~\ref{thm:main}. We first need a much weaker version which says that
the 0th homotopy group of the $\mathbb{E}_2$-Thom spectrum in question, as a
ring, is isomorphic to $R$:

\begin{lemma}
  \label{lem:pi0Mf}The 0th homotopy group \ $\pi_0 (M f_{R, \xi})$ of the Thom
  spectrum associated to $f_{R, \xi}$ is isomorphic to $R$ for any perfectoid
  ring $R$.
\end{lemma}

\begin{proof}
  We mimic a segment of the proof of Theorem~A.1 in {\cite{Krause}}:
  
  We note that $M f_{R, \xi}$ is connective, so we have
  \[ \pi_0 (M f_{R, \xi}) \cong \pi_0 (W^+ (R^{\flat})_{h \Omega^3 S^3}) \cong
     \pi_0 (W^+ (R^{\flat}))_{\pi_0 (\Omega^3 S^3)} \]
  where the $\pi_0 (\Omega^3 S^3) \cong \mathbb{Z}$-action on $\pi_0 (W^+
  (R^{\flat})) \cong W (R^{\flat})$ is given by multiplication by $1 - \xi$,
  hence
  \[ \pi_0 (W^+ (R^{\flat}))_{\pi_0 (\Omega^3 S^3)} \cong W (R^{\flat}) / (1 -
     (1 - \xi)) \cong R \]
\end{proof}

In view of Lemma~\ref{lem:pi0Mf}, in order to prove Theorem~\ref{thm:main},
it suffices to show that

\begin{proposition}
  \label{prop:main}The 0th Postnikov section $t_{R, \xi} \of M f_{R, \xi}
  \rightarrow \tau_{\leq 0} M f_{R, \xi} \simeq H R$, being an
  $\mathbb{E}_2$-map a priori, is an equivalence of spectra.
\end{proposition}

To begin with, we first note that the special case when $R$ is a perfect
$\mathbb{F}_p$-algebra is already covered by previous considerations:

\begin{lemma}
  \label{lem:eq-postnikov-perf}The $t_{R, \xi}$ in question is an equivalence
  of spectra when $R$ is a perfect $\mathbb{F}_p$-algebra.
\end{lemma}

\begin{proof}
  Theorem~\ref{thm:perf-thom} tells us that there is an equivalence $M f_{R,
  \xi} \rightarrow H R$. The lemma follows from the fact that $H R$ lives in
  $(\tmop{Mod}_{W^+ (R)})_{\leq 0}$ and that the 0th Postnikov section is
  essentially unique.
\end{proof}

We first note that both $M f_{R, \xi}$ and $H R$ admit canonical $W^+
(R^{\flat})$-module structures. Our strategy breaks up into several steps:
\begin{enumeratenumeric}
  \item Prove some finiteness and completeness results of $M f_{R, \xi}$ and
  $H R$ as $W^+ (R^{\flat})$-modules;
  
  \item Show that $t_{R, \xi}$ is an equivalence after the base change along
  $W^+ (R^{\flat}) \rightarrow W^+ (\kappa)$, and hence an equivalence after a
  further base change along $W^+ (\kappa) \rightarrow H \kappa$ to the special
  fiber $H \kappa$;
  
  \item The composition $W^+ (R^{\flat}) \rightarrow W^+ (\kappa) \rightarrow
  H \kappa$ coincides with the composition $W^+ (R^{\flat}) \rightarrow H
  R^{\flat} \rightarrow H \kappa$, and a Nakayama-like argument shows that
  $t_{R, \xi}$ is an equivalence after base change along $W^+ (R^{\flat})
  \rightarrow H R^{\flat}$;
  
  \item Deduce that $t_{R, \xi}$ is an equivalence by completeness.
\end{enumeratenumeric}
To proceed, by Corollary~\ref{cor:image-xi}, we also fix an invertible element
$u \in \tmop{GL}_1 (W (\kappa))$ associated to $\xi$ so that the image of
$\xi$ in $W (\kappa)$ is $pu$.

\subsection{Finiteness and completeness of $M f_{R, \xi}$ and $H R$ as $W^+
(R^{\flat})$-modules}

\begin{lemma}
  \label{lem:witt-aperf}$H W (k)$ is an almost perfect $W^+ (k)$-module for
  any perfect $\mathbb{F}_p$-algebra $k$.
\end{lemma}

\begin{proof}
  If $k =\mathbb{F}_p$, then $W^+ (\mathbb{F}_p) \simeq \mathbb{S}_p^{\wedge}$
  is a coherent ring as in Definition~\ref{def:coherentE1}, and $H W
  (\mathbb{F}_p) \simeq H\mathbb{Z}_p \simeq H \pi_0 (W^+ (\mathbb{F}_p))$ is
  an almost perfect $\mathbb{S}_p^{\wedge}$-module by
  Corollary~\ref{cor:coherent_aperf}.
  
  In general, by Proposition~\ref{prop:sphwitt}, we have $H W (k) \simeq W^+
  (k) \otimes_{\mathbb{S}_p^{\wedge}} H\mathbb{Z}_p$, hence $H W (k)$ is
  almost perfect by Proposition~\ref{prop:aperf-base-change}.
\end{proof}

\begin{corollary}
  \label{cor:hr-aperf}$H R$ is an almost perfect $W^+ (R^{\flat})$-module.
\end{corollary}

\begin{proof}
  $H R$ is the cofiber of the multiplication map $m_{\xi} \of H W (R^{\flat})
  \rightarrow H W (R^{\flat})$ where the domain and the codomain are almost
  perfect (Lemma~\ref{lem:witt-aperf}), hence $H R$ is also almost perfect
  (Proposition~\ref{prop:aperf}).
\end{proof}

We need a nontrivial input from algebraic topology:

\begin{proposition}
  \label{prop:finite-Kan-double-loop}There exists a Kan complex $X_{\bullet}$
  which is homotopy equivalent to the double loop space $\Omega^2 S^3$ of the
  3-sphere such that $X_n$ is a finite set for each $[n] \in
  \Delta^{\tmop{op}}$.
\end{proposition}

\begin{proof}
  This is essentially due to {\cite[Thm~A and~B]{Wall1965}} and Serre. We
  first note that, the loop space $\Omega^2 S^3$ is a loop space therefore
  simple {\cite[Cor~1.4.5]{May2012}}. Now we show that $\Omega^2 S^3$ is of
  finite type, i.e. homotopy equivalent to a CW-complex with finite skeleta.
  By {\cite[Thm~4.5.2]{May2012}}, it suffices to show that $H_i (\Omega^2 S^3
  ; \mathbb{Z})$ are finitely generated for all $i \in \mathbb{N}_{> 0}$. The
  argument is standard (due to Serre): we know that $H_i (S^3 ; \mathbb{Z})$
  are finitely generated for all $i \in \mathbb{N}$. Applying
  {\cite[Thm~20.4.1]{tomDieck2008}} to the fiber sequence $\Omega S^3
  \rightarrow \mathord{\ast} \rightarrow S^3$ in $\mathcal{S}$, we deduce that
  $H_i (\Omega S^3)$ are finitely generated for all $i \in \mathbb{N}$. We
  apply again {\cite[Thm~20.4.1]{tomDieck2008}} to the fiber sequence
  $\Omega^2 S^3 \rightarrow \mathord{\ast} \rightarrow \Omega S^3$, we deduce
  that $H_i (\Omega^2 S^3)$ are finitely generated. Now the result follows
  from the simplicial approximation theorem.
\end{proof}

\begin{lemma}
  \label{lem:mf-aperf}$M f_{R, \xi}$ is an almost perfect $W^+
  (R^{\flat})$-module.
\end{lemma}

\begin{proof}
  We first pick up a Kan complex $X_{\bullet}$ representing $\Omega^2 S^3$
  where each $X_n$ is a finite set as in
  Proposition~\ref{prop:finite-Kan-double-loop}. It follows from Bousfield-Kan
  formula (see, for example, Corollary~12.3 in {\cite{Shah2018}}) that $M
  f_{R, \xi}$ could be written as the geometric realization of a simplicial
  object $N_{\bullet}$ where each $N_n$ is a free $W^+ (R^{\flat})$-module of
  finite rank, hence almost perfect by Proposition~\ref{prop:aperf}.
\end{proof}

\begin{corollary}
  \label{cor:cofib-tr-aperf}$\tmop{cofib} (t_{R, \xi})$ is an almost perfect
  $W^+ (R^{\flat})$-module.
\end{corollary}

\begin{proof}
  The subcategory of almost perfect modules are closed under taking cofibers
  and base changes (Proposition~\ref{prop:aperf}). The statement then follows
  from Corollary~\ref{cor:hr-aperf} and Lemma~\ref{lem:mf-aperf}.
\end{proof}

\begin{lemma}
  \label{lem:HR-p-compl}The spectrum $H R$ is $p$-complete.
\end{lemma}

\begin{proof}
  By definition of perfectoid rings, $R$ is $p$-adically complete, therefore
  $H R$ is $p$-complete by Proposition~\ref{prop:I-adic}.
\end{proof}

\begin{lemma}
  \label{lem:thom-p-compl}The spectrum $M f_{R, \xi}$ is $p$-complete.
\end{lemma}

\begin{proof}
  We note that $W^+ (R^{\flat})$ is $p$-complete by definition of spherical
  Witt vectors, and $M f_{R, \xi}$ is almost perfect, therefore $p$-complete
  by Proposition~\ref{prop:aperf-complete}.
\end{proof}

\begin{corollary}
  \label{cor:cofib-tr-p-compl}The spectrum $\tmop{cofib} (t_{R, \xi})$ is
  $p$-complete.
\end{corollary}

\begin{proof}
  It follows from Corollary~\ref{cor:cofib-tr-aperf} and
  Proposition~\ref{prop:aperf-complete}.
\end{proof}

\subsection{$t_{R, \xi}$ is an equivalence after the base change along $W^+
(R^{\flat}) \rightarrow W^+ (\kappa)$}

The proof is similar to that of Theorem~\ref{thm:perf-thom}, except that we
need to be more careful to identify the maps.

\begin{lemma}
  \label{lem:eq-thom-base-change}There is a canonical equivalence $M
  f_{\kappa, pu} \xrightarrow{\simeq} W^+ (\kappa) \otimes_{W^+ (R^{\flat})} M
  f_{R, \xi}$ of $W^+ (\kappa)$-modules.
\end{lemma}

\begin{proof}
  We first note that the image of the multiplication map $m_{1 - \xi} \of W^+
  (R^{\flat}) \rightarrow W^+ (R^{\flat})$ under the base change functor $W^+
  (\kappa) \otimes_{W^+ (R^{\flat})} - \of \tmop{Mod}_{W^+ (R^{\flat})}
  \rightarrow \tmop{Mod}_{W^+ (\kappa)}$ is the multiplication map $m_{1 - pu}
  \of W^+ (\kappa) \rightarrow W^+ (\kappa)$.
  
  Therefore $f_{\kappa, pu}$ coincides with the composition
  \[ \Omega^2 S^3 \xrightarrow{f_{R, \xi}} \tmop{BGL}_1 (W^+ (R^{\flat}))
     \xrightarrow{W^+ (\kappa) \otimes_{W^+ (R^{\flat})} -} \tmop{BGL}_1 (W^+
     (\kappa)) \]
  Along with the fact that the functor $W^+ (\kappa) \otimes_{W^+ (R^{\flat})}
  - \of \tmop{Mod}_{W^+ (R^{\flat})} \rightarrow \tmop{Mod}_{W^+ (\kappa)}$
  commutes with small colimits, or to be more precise, that the natural
  transformation $\tmop{colim}_i  (W^+ (\kappa) \otimes_{W^+ (R^{\flat})} M_i)
  \rightarrow W^+ (\kappa) \otimes_{W^+ (R^{\flat})} (\tmop{colim}_i M_i)$ is
  an equivalence for any diagram $(M_i)_i$ in $\tmop{Mod}_{W^+ (R^{\flat})}$,
  we deduce that there is a canonical equivalence $M f_{\kappa, pu}
  \xrightarrow{\simeq} W^+ (\kappa) \otimes_{W^+ (R^{\flat})} M f_{R, \xi}$ as
  $W^+ (\kappa)$-modules.
\end{proof}

\begin{lemma}
  \label{lem:POkK}Given a morphism of perfect $\mathbb{F}_p$-algebras $k
  \rightarrow K$, the commutative diagram of $\mathbb{E}_{\infty}$-rings
  \[ \begin{array}{ccc}
       W^+ (k) & \longrightarrow & W^+ (K)\\
       \longdownarrow &  & \longdownarrow\\
       H W (k) & \longrightarrow & H W (K)
     \end{array} \]
  is a pushout square.
\end{lemma}

\begin{proof}
  Consider the commutative diagram of $\mathbb{E}_{\infty}$-rings
  \[ \begin{array}{ccccc}
       \mathbb{S}_p^{\wedge} & \longrightarrow & W^+ (k) & \longrightarrow &
       W^+ (K)\\
       \longdownarrow &  & \longdownarrow &  & \longdownarrow\\
       H\mathbb{Z}_p & \longrightarrow & H W (k) & \longrightarrow & H W (K)
     \end{array} \]
  By Proposition~\ref{prop:sphwitt}, we know that the left square and the
  outer square are pushout squares, therefore so is the right square.
\end{proof}

\begin{lemma}
  \label{lem:eq-HR-base-change}There is a canonical equivalence $W^+ (\kappa)
  \otimes_{W^+ (R^{\flat})} H R \rightarrow H \kappa$ of $W^+
  (\kappa)$-modules.
\end{lemma}

\begin{proof}
  Combining two pushout squares in the category of
  $\mathbb{E}_{\infty}$-rings:
  \[ \begin{array}{ccc}
       W^+ (R^{\flat}) & \longrightarrow & W^+ (\kappa)\\
       \longdownarrow & \sigma & \longdownarrow\\
       H W (R^{\flat}) & \longrightarrow & H W (\kappa)\\
       \longdownarrow & \tau & \longdownarrow\\
       H R & \longrightarrow & H \kappa
     \end{array} \]
  where $\sigma$ is a pushout square by Lemma~\ref{lem:POkK} and $\tau$ is a
  pushout square by Proposition~\ref{prop:perfdPOfiber}.
\end{proof}

\begin{lemma}
  \label{lem:eq-base-change-sphwitt-kappa}The map $W^+ (\kappa) \otimes_{W^+
  (R^{\flat})} t_{R, \xi} \of W^+ (\kappa) \otimes_{W^+ (R^{\flat})} M f_{R,
  \xi} \rightarrow W^+ (\kappa) \otimes_{W^+ (R^{\flat})} H R$ is equivalent
  to $t_{\kappa, pu} \of M f_{\kappa, pu} \rightarrow H \kappa$.
\end{lemma}

\begin{proof}
  In view of Lemma~\ref{lem:eq-thom-base-change} and
  Lemma~\ref{lem:eq-HR-base-change}, we only need to show that $t_{\kappa, pu}
  \of M f_{\kappa, pu} \rightarrow H \kappa$ coincides with the composition of
  the equivalences $M f_{\kappa, pu} \rightarrow W^+ (\kappa) \otimes_{W^+
  (R^{\flat})} M f_{R, \xi}$, $W^+ (\kappa) \otimes_{W^+ (R^{\flat})} t_{R,
  \xi}$ and $W^+ (\kappa) \otimes_{W^+ (R^{\flat})} H R \rightarrow H \kappa$.
  In other words, it suffices to show that the composition in question is the
  0th Postnikov section. We only need to check that the composition induces an
  isomorphism on $\pi_0$ by basic properties of $t$-structures, since
  $\tau_{\leq 0} M f_{\kappa, pu} \simeq H \kappa$. It suffices to show that
  $W^+ (\kappa) \otimes_{W^+ (R^{\flat})} t_{R, \xi}$ induces an isomorphism
  on $\pi_0$, and this follows from the fact that all spectra in question are
  connective and that $t_{R, \xi}$ induces an isomorphism on $\pi_0$ by
  Lemma~\ref{lem:pi0Mf}.
\end{proof}

\begin{corollary}
  \label{cor:eq-base-change-kappa}$H \kappa \otimes_{W^+ (R^{\flat})} t_{R,
  \xi}$ is an equivalence of spectra.
\end{corollary}

\begin{proof}
  It follows from Lemma~\ref{lem:eq-base-change-sphwitt-kappa} and
  Lemma~\ref{lem:eq-postnikov-perf}.
\end{proof}

\subsection{$t_{R, \xi}$ is an equivalence after the base change along $W^+
(R^{\flat}) \rightarrow H R^{\flat}$}

\begin{lemma}
  \label{lem:nakayama-like}Let $M$ be an $H R^{\flat}$-module which is bounded
  below and almost perfect. If there exists an $r \in \mathbb{N}$ such that
  $\xi^r \pi_n (M) = 0$ for all $n \in \mathbb{Z}$, and $H \kappa \otimes_{H
  R^{\flat}} M \simeq 0$, then $M \simeq 0$.
\end{lemma}

\begin{proof}
  We show inductively on $n$ that $\pi_n M = 0$.
  \begin{itemizedot}
    \item Since $M$ is bounded below, $\pi_n M = 0$ for $n \ll 0$;
    
    \item Suppose that for $m < n$ we have $\pi_m M = 0$. Then by unrolling
    Definition \ref{def:aperf}, $\pi_n M$ is a compact object in the category
    of discrete $R^{\flat}$-modules, therefore is finitely presented and in
    particular finitely generated. Now we have
    \[ 0 = \pi_n (H \kappa \otimes_{H R^{\flat}} M) = \tmop{Tor}_0^{R^{\flat}}
       (\kappa, \pi_n M) . \]
    By Proposition \ref{prop:ker-fiber} and that $\xi^r \pi_n (M) = 0$, we
    have
    \[ \tmop{Tor}_0^{R^{\flat}} (\kappa, \pi_n M) = \tmop{Tor}_0^{R^{\flat} /
       \xi^r} (\kappa, \pi_n M) \]
    We note that $\ker (R^{\flat} / \xi^r \rightarrow \kappa)$ lies in the
    (nil-radical, therefore) Jacobson radical of $R^{\flat} / \xi^r$, thus
    $\pi_n M = 0$, by Nakayama's lemma along with the fact that $\pi_n M$ is
    finitely generated.
  \end{itemizedot}
\end{proof}

\begin{remark}
  Matthew Morrow told us that in Lemma~\ref{lem:nakayama-like}, the hypothesis
  $\xi^r \pi_n (M) = 0$ is redundant, since the kernel of the map $R^{\flat}
  \rightarrow \kappa$ of $\mathbb{F}_p$-algebras lies in the radical of the
  ideal $\xi R^{\flat} \subseteq \tmop{Rad} (R^{\flat})$ where $\tmop{Rad}
  (R^{\flat})$ is the Jacobson radical of the $\mathbb{F}_p$-algebra
  $R^{\flat}$ and the inclusion $\xi R^{\flat} \subseteq \tmop{Rad}
  (R^{\flat})$ is deduced from the $\xi R^{\flat}$-adically completeness of
  the $\mathbb{F}_p$-algebra $R^{\flat}$. Since the Jacobson radical is
  ``radical'', the kernel of the map $R^{\flat} \rightarrow \kappa$ also lies
  in the Jacobson radical $\tmop{Rad} (R^{\flat})$. We decide to preserve the
  original version to reflect our real thoughts.
\end{remark}

\begin{corollary}
  \label{cor:eq-base-change-hrflat}$H R^{\flat} \otimes_{W^+ (R^{\flat})}
  t_{R, \xi}$ is an equivalence of spectra.
\end{corollary}

\begin{proof}
  Note that
  \[ \pi_0 (H R^{\flat} \otimes_{W^+ (R^{\flat})} M f_{R, \xi}) =
     \tmop{Tor}_0^{W (R^{\flat})} (R^{\flat}, \pi_0 (M f_{R, \xi})) =
     R^{\flat} / \xi R^{\flat} \]
  and
  \[ \pi_0 (H R^{\flat} \otimes_{W^+ (R^{\flat})} H R) = \tmop{Tor}_0^{W
     (R^{\flat})} (R^{\flat}, R) = R^{\flat} / \xi R^{\flat} \]
  and that $H R^{\flat} \otimes_{W^+ (R^{\flat})} M f_{R, \xi}$, $H R^{\flat}
  \otimes_{W^+ (R^{\flat})} H R$ are connective $\mathbb{E}_{\infty}$-rings,
  we conclude that the homotopy groups of these $\mathbb{E}_{\infty}$-rings
  are $\xi$-torsion groups, which implies that for all $n \in \mathbb{Z}$,
  \[ \xi^2 \pi_n (\tmop{cofib} (H R^{\flat} \otimes_{W^+ (R^{\flat})} t_{R,
     \xi})) = 0 \]
  In addition, since the subcategory of almost perfect modules are closed
  under base changes (Proposition~\ref{prop:aperf-base-change}), we deduce
  from Corollary~\ref{cor:cofib-tr-aperf} that $\tmop{cofib} (H R^{\flat}
  \otimes_{W^+ (R^{\flat})} t_{R, \xi}) \simeq H R^{\flat} \otimes_{W^+
  (R^{\flat})} \tmop{cofib} (t_{R, \xi})$ is almost perfect. On the other
  hand, being the cofiber of a map of connective spectra, it is also
  connective. Then we invoke Lemma~\ref{lem:nakayama-like} along with
  Corollary~\ref{cor:eq-base-change-kappa} to conclude that $\tmop{cofib} (H
  R^{\flat} \otimes_{W^+ (R^{\flat})} t_{R, \xi}) \simeq 0$.
\end{proof}

\subsection{Conclude: $t_{R, \xi}$ is an equivalence}

We are now at the final stage to conclude a proof of
Proposition~\ref{prop:main}, and consequently, Theorem~\ref{thm:main}.

\begin{proof*}{Proof of Proposition~\ref{prop:main}}
  We recall that by Theorem~\ref{thm:exist_thickening} and
  Example~\ref{ex:sphwitt}, there is a pushout square of
  $\mathbb{E}_{\infty}$-rings:
  \[ \begin{array}{ccc}
       \mathbb{S}_p^{\wedge} & \longrightarrow & W^+ (R^{\flat})\\
       \longdownarrow &  & \longdownarrow\\
       H\mathbb{F}_p & \longrightarrow & H R^{\flat}
     \end{array} \]
  Therefore by Corollary~\ref{cor:eq-base-change-hrflat} we have
  \[ 0 \simeq \tmop{cofib} (H R^{\flat} \otimes_{W^+ (R^{\flat})} t_{R, \xi})
     \simeq H R^{\flat} \otimes_{W^+ (R^{\flat})} \tmop{cofib} (t_{R, \xi})
     \simeq H\mathbb{F}_p \otimes_{\mathbb{S}_p^{\wedge}} \tmop{cofib} (t_{R,
     \xi}) \]
  We then invoke Corollary~\ref{cor:compl-mod-p} with
  Corollary~\ref{cor:cofib-tr-p-compl} to deduce that $\tmop{cofib} (t_{R,
  \xi}) \simeq 0$.
\end{proof*}

\subsection{An intermezzo: Identifying $\tmop{THH} \left( - / W^+ (k) \right)
$ with $\tmop{THH} \left( - \right) $ after $p$-completion}\label{subsec:THH}

In this subsection, we will show that Proposition~\ref{prop:compl-thh} follows
from Proposition~\ref{prop:thh-over-sphwitt}. It suffices to prove the
following lemma:

\begin{lemma}
  Let $R$ be an $\mathbb{E}_1$-algebra over $W^+ (k)$ where $k$ is a perfect
  $\mathbb{F}_p$-algebra. Then the canonical map $\tmop{THH} (R) \rightarrow
  \tmop{THH} (R / W^+ (k))$ induced by $\mathbb{S} \rightarrow
  \mathbb{S}_p^{\wedge}$ is an equivalence after $p$-completion.
\end{lemma}

\begin{proof}
  Note that $\tmop{THH} (R / W^+ (k)) \simeq W^+ (k) \otimes_{\tmop{THH} (W^+
  (k))} \tmop{THH} (R)$. We are left to show that the canonical map
  $\tmop{THH} (W^+ (k)) \rightarrow W^+ (k)$ is an equivalence after
  $p$-completion. In view of Corollary~\ref{cor:compl-mod-p}, we only need to
  check it after tensoring with $H\mathbb{F}_p$. We note that the base changed
  map $H\mathbb{F}_p \otimes \tmop{THH} (W^+ (k)) \rightarrow H\mathbb{F}_p
  \otimes W^+ (k)$ fits into the commutative diagram
  \[ \begin{array}{ccc}
       H\mathbb{F}_p \otimes \tmop{THH} (W^+ (k)) & \longrightarrow &
       H\mathbb{F}_p \otimes W^+ (k)\\
       \longdownarrow & \sigma & \longdownequal\\
       \tmop{THH} (H\mathbb{F}_p \otimes W^+ (k) / H\mathbb{F}_p) &
       \longrightarrow & H\mathbb{F}_p \otimes W^+ (k)\\
       \longdownarrow & \tau & \longdownarrow\\
       \tmop{THH} (H k / H\mathbb{F}_p) & \longrightarrow & H k
     \end{array} \]
  where the commutativity of $\sigma$ follows from the functoriality of the
  base change functor of $\tmop{THH}$, and the commutative of $\tau$ follows
  from the functoriality of the natural transformation $\tmop{THH} \left( - /
  H\mathbb{F}_p \right) \rightarrow (-)$. All vertical maps are equivalences
  of spectra: the upper left map is the base change equivalence, and the lower
  right map is the equivalence by Proposition~\ref{prop:sphwitt}, and the
  lower left map is the image of this equivalence under the functor
  $\tmop{THH} \left( - / H\mathbb{F}_p \right)$ and hence also an equivalence.
  The bottom horizontal map is an equivalence by the fact that $k$ is a
  perfect $\mathbb{F}_p$-algebra.
\end{proof}

\section{Analogues}\label{sec:analog}

It is worth to note that in Bhatt and Scholze's recent work
{\cite{Bhatt2019}}, they introduced the concept of prisms $(A, I)$ which
serves as a ``non-perfect'' version of perfectoid rings. Especially, the
category of perfect prisms $(A, I)$ is equivalent to that of perfectoid rings
$A / I$, and given a perfectoid ring $R$, the corresponding perfect prism is
given by $(W (R^{\flat}), \ker \theta)$. It is interesting to know whether we
can generalize our description for general orientable prisms $(A, I)$, that is
to say,

\begin{question}
  \label{ques:prism-thom}Given an orientable prism $(A, I = (d))$ . When can
  we find an $\mathbb{E}_{\infty}$-ring spectrum $A^+$ (which satisfies some
  hypotheses related to $A$. A naive guess would be that $\pi_0 (A^+) = A$)
  and a map $\Omega^2 S^3 \rightarrow \tmop{BGL}_1 (A^+)$ to which the
  associated $\mathbb{E}_2$-Thom spectrum (possibly after $p$-completion)
  coincides with $A / I$.
\end{question}

We don't know the answer in this generality. However, we will discuss another
special class of prism (related to Breuil-Kisin cohomology) for which an
analogue holds. This result is more-or-less known by experts. In fact, it is
essentially equivalent to Remark~3.4 in {\cite{Krause2019}} of which no proof
is presented. In this section, we will first recall some basic facts about
complete discrete valuation rings, then we will indicate briefly how to adapt
our proof above to this special class.

\subsection{Preparations}

\begin{definition}[{\cite[Section I.1]{Serre1979}}]
  A ring $A$ is called a {\tmdfn{discrete valuation ring}}, or a
  {\tmdfn{DVR}}, if it is a principal ideal domain that has a unique non-zero
  prime ideal $\mathfrak{m}$. In this case, since $A$ is local, we also denote
  the DVR $A$ by $(A, \mathfrak{m})$. The field $A /\mathfrak{m}$ is called
  the {\tmdfn{residue field}} of $A$. A generator of $\mathfrak{m}$, unique up
  to multiplication by an invertible element, is called a
  {\tmdfn{uniformizer}}, usually denoted by $\varpi$.
\end{definition}

\begin{definition}
  A DVR $(A, \mathfrak{m})$ is called {\tmdfn{of mixed characteristics $(0,
  p)$}} if the field of fraction $\tmop{Frac} (A)$ of $A$ is of
  characteristics $0$ while the residue field $A /\mathfrak{m}$ is of
  characteristics $p$, which implies that $0 \neq p \in \mathfrak{m}$.
\end{definition}

\begin{definition}[{\cite[Section~I.1]{Serre1979}}]
  Let $(A, \mathfrak{m})$ be a DVR. The {\tmdfn{valuation}} of an element $x
  \in A \setminus 0$ is defined to be the maximal integer $n \in \mathbb{N}$
  such that $x \in \mathfrak{m}^n$, which always exists, denoted by $v (x) \in
  \mathbb{N}$.
\end{definition}

\begin{definition}[{\cite[Section~II.5]{Serre1979}}]
  Let $(A, \mathfrak{m})$ be a DVR of mixed characteristics $(0, p)$. Then the
  integer $e = v (p)$ is called the {\tmdfn{absolute ramification index}} of
  $A$.
\end{definition}

\begin{definition}[{\cite[Chapter~II]{Serre1979}}]
  A DVR $(A, \mathfrak{m})$ is called {\tmdfn{complete}} if it is complete
  with respect to the $\mathfrak{m}$-adic topology, that is to say, the
  canonical map from $A$ to the limit of the tower
  \[ \cdots \rightarrow A /\mathfrak{m}^n \rightarrow \cdots \rightarrow A
     /\mathfrak{m}^2 \rightarrow A /\mathfrak{m} \]
  is an isomorphism.
\end{definition}

\begin{proposition}[{\cite[Section~II.5, Theorem~4]{Serre1979}}]
  \label{prop:struct-thm-cdvr-mixed}Let $(A, \mathfrak{m})$ be a complete DVR
  of mixed characteristics $(0, p)$ with residue field $k$ being perfect. Let
  $e$ be its absolute ramification index. Let $\varpi \in \mathfrak{m}$ be a
  uniformizer. Then there exists an Eisenstein $W (k)$-polynomial $E (u) \in W
  (k) [u]$ (that is, a $W (k)$-polynomial $E (u) = u^e + \sum_{j = 0}^{e - 1}
  a_j u^j$ such that $p \divides a_j$ for $j = 0, \ldots, e - 1$ and $p^2
  \ndivides a_0$, where $W (k)$ is the ring of Witt vectors as before) along
  with an isomorphism $W (k) [u] / (E (u)) \xrightarrow{\sim} A$ which maps
  $u$ to the uniformizer $\varpi \in \mathfrak{m}$.
\end{proposition}

In the rest of this section, we will fix a complete DVR $(A, \mathfrak{m})$ of
mixed characteristics $(0, p)$ with residue field $k$ being perfect, absolute
ramification index $e$ and a uniformizer $\varpi \in \mathfrak{m}$. We also
fix a choice of an Eisenstein $W (k)$-polynomial $E (u) \in W (k) [u]$ as in
Proposition~\ref{prop:struct-thm-cdvr-mixed}. We first note that

\begin{proposition}
  The element $1 - E (u) \in W (k) [[u]]$ is invertible.
\end{proposition}

\begin{proof}
  Write $E (u) = u^e + \sum_{j = 0}^{e - 1} a_j u^j$ as in
  Proposition~\ref{prop:struct-thm-cdvr-mixed}. Note that $W (k)$ is
  $p$-adically complete, therefore $1 - a_0$ is invertible in $W (k)$, which
  implies that $1 - E (u) \in W (k) [[u]]$ is invertible.
\end{proof}

Let $W^+ (k) [u]$ be the ``single variable polynomial $W^+ (k)$-algebra'',
that is, the $\mathbb{E}_{\infty}$-$W^+ (k)$-algebra $W^+ (k)
\otimes_{\mathbb{S}} \mathbb{S} [\mathbb{N}]$. Since the space $\mathbb{N}$ is
endowed with discrete topology, we have

\begin{proposition}
  \label{prop:single-var-poly-homotopy-grps}As a $W^+ (k)$-module, $W^+ (k)
  [u]$ is equivalent to the direct sum $\bigoplus_{j = 0}^{\infty} u^j W^+
  (k)$, a free $W^+ (k)$-module. The graded homotopy group $\pi_{\ast} (W^+
  (k) [u])$, as a (graded-commutative) $\pi_{\ast} (W^+ (k))$-algebra, is
  equivalent to $\pi_{\ast} (W^+ (k)) [u]$, where $\deg u = 0$.
\end{proposition}

Now let $W^+ (k) [[u]]$ be the $(u)$-completion of the
$\mathbb{E}_{\infty}$-$W^+ (k)$-algebra $W^+ (k) [u]$. To study $W^+ (k)
[[u]]$, we need some preparations.

\begin{proposition}
  \label{prop:cofiber-mult}Let $n \in \mathbb{N}$ be a natural number. Let
  $m_{u^n} \of W^+ (k) [u] \rightarrow W^+ (k) [u]$ be the multiplication map
  given by $u^n \in \pi_0 (W^+ (k) [u]) = W (k) [u]$. Then the $W^+ (k)
  [u]$-module $\tmop{cofib} (m_{u^n})$ as a $W^+ (k)$-module is a free $W^+
  (k)$-module $\bigoplus_{j = 0}^{n - 1} u^j W^+ (k)$ of rank $n$, which
  admits an $\mathbb{E}_{\infty}$-$W^+ (k) [u]$-algebra structure. In
  particular, we have the cofiber sequence
  \[ W^+ (k) [u] \xrightarrow{m_u} W^+ (k) [u] \rightarrow W^+ (k) \]
  of $W^+ (k) [u]$-modules.
\end{proposition}

\begin{proof}
  For any space $X \in \mathcal{S}$, we let $X_+ \in \mathcal{S}_{\ast}$
  denote the pointed discrete space $\{ \ast \} \cup X$. Especially,
  $\mathbb{N}_+ = \{ \ast \} \cup \mathbb{N}$ and $(\mathbb{N}_{< n})_+ = \{
  \ast \} \cup \mathbb{N}_{< n}$. The addition map $\mathbb{N} \rightarrow
  \mathbb{N}, m \mapsto n + m$ induces a map of pointed spaces $\alpha_n \of
  \mathbb{N}_+ \rightarrow \mathbb{N}_+$. Note that in the $\infty$-category
  $\mathcal{S}$ of spaces, we have a pushout diagram
  \[ \begin{array}{ccc}
       \mathbb{N}_+ & \xrightarrow{\alpha_n} & \mathbb{N}_+\\
       \longdownarrow &  & \longdownarrow\\
       \{ \ast \} & \longrightarrow & (\mathbb{N}_{< n})_+
     \end{array} \]
  to which we apply the functor $\Sigma^{\infty} \of \mathcal{S}_{\ast}
  \rightarrow \tmop{Sp}$, left adjoint of the functor $\Omega_{\ast}^{\infty}
  \of \tmop{Sp} \rightarrow S_{\ast}$ therefore preserving colimits, we get a
  cofiber sequence $\mathbb{S} [u] \xrightarrow{u^n} \mathbb{S} [u]
  \rightarrow \bigoplus_{j = 0}^{n - 1} u^j \mathbb{S}$. A further base change
  to $W^+ (k)$ gives rise to the result. In addition, the multiplication
  structure could be seen from the fact that the addition map $\mathbb{N}
  \rightarrow \mathbb{N}, m \mapsto n + m$ in fact defines a monoidal action.
\end{proof}

\begin{corollary}
  \label{cor:homotopy-grp-cofiber-mult}Let $n \in \mathbb{N}$ be a natural
  number. Let $m_{u^n} \of W^+ (k) [u] \rightarrow W^+ (k) [u]$ be the
  multiplication map. Then homotopy groups $\pi_{\ast} (\tmop{cofib}
  (m_{u^n}))$ of the cofiber as $\pi_{\ast} (W^+ (k))$ could be identified
  with $\pi_{\ast} (W^+ (k)) [u] / (u^n)$, and the long exact sequence of
  homotopy groups associated to the cofiber sequence $W^+ (k) [u]
  \xrightarrow{m_{u^n}} W^+ (k) [u] \rightarrow \tmop{cofib} (m_{u^n})$
  decomposes as short exact sequences, which assemble to a short exact
  sequence of graded $\pi_{\ast} (W^+ (k)) [u]$-modules:
  \[ 0 \rightarrow \pi_{\ast} (W^+ (k)) [u] \xrightarrow{u^n} \pi_{\ast} (W^+
     (k)) [u] \rightarrow \pi_{\ast} (W^+ (k)) [u] / (u^n) \rightarrow 0 \]
  Furthermore, this sequence is functorial in $n \in (\mathbb{N}, >)$.
\end{corollary}

\begin{proposition}
  \label{prop:single-var-power-series-homotopy-grps}The
  $\mathbb{E}_{\infty}$-$W^+ (k)$-algebra $W^+ (k) [[u]]$ is connective. The
  zeroth homotopy group of $\pi_0 (W^+ (k) [[u]])$ is isomorphic to the
  $(u)$-adic completion of the polynomial $W (k)$-algebra $W (k) [u]$, that
  is, the formal power series $W (k)$-algebra $W (k) [[u]]$, as $W
  (k)$-algebras.
\end{proposition}

Our proof is incomplete: we only identify the $W (k)$-module structures on
homotopy groups. A formal identification of algebra structures would require
more rudiments about the symmetric monoidal structure on the completion
functor than we know.

\begin{proof}
  We reinterpret Proposition~\ref{prop:calc-compl} as follows: since the limit
  functor is exact, it commutes with cofibers, therefore we can rewrite $W^+
  (k) [[u]] = (W^+ (k) [u])_{(u)}^{\wedge}$ as the limit of the tower
  \[ \cdots \rightarrow \tmop{cofib} \left( W^+ (k) [u] \xrightarrow{u^2} W^+
     (k) [u] \right) \rightarrow \tmop{cofib} \left( W^+ (k) [u]
     \xrightarrow{u} W^+ (k) [u] \right) \]
  After passage to homotopy groups, by
  Corollary~\ref{cor:homotopy-grp-cofiber-mult}, we get the tower of graded
  $\pi_{\ast} (W^+ (k)) [u]$-modules
  \begin{equation}
    \cdots \rightarrow \pi_{\ast} (W^+ (k)) [u] / (u^n) \rightarrow \cdots
    \rightarrow \pi_{\ast} (W^+ (k)) [u] / (u^2) \rightarrow \pi_{\ast} (W^+
    (k)) [u] / (u) \label{eq:tower-homotopy}
  \end{equation}
  which is degree-wise a tower of surjective maps. It follows from Milnor's
  sequence that the graded $\pi_{\ast} (W^+ (k)) [u]$-module ${\pi_{\ast}}_{}
  (W^+ (k) [[u]])$ is isomorphic to the (ordinary) inverse limit of the tower
  \eqref{eq:tower-homotopy}, that is, $\pi_{\ast} (W^+ (k)) [[u]]$. Take $\ast
  = 0$, we get the result.
\end{proof}

The following lemma serves as a key tool in our proof:

\begin{lemma}
  \label{lem:compl-mod-u}Let $M$ be a $W^+ (k) [u]$-(or $W^+ (k)
  [[u]]$-)module (spectrum). If the spectrum $W^+ (k) \otimes_{W^+ (k) [u]} M$
  (or $W^+ (k) \otimes_{W^+ (k) [[u]]} M$ respectively) is contractible, then
  so is the $(u)$-completion of the spectrum $M$. In particular, if
  furthermore $W^+ (k) [u]$-(or $W^+ (k) [[u]]$-)module $M$ is assumed to be
  $(u)$-complete, then the spectrum $M$ is contractible.
\end{lemma}

\begin{proof}
  We first assume that the spectrum $W^+ (k) \otimes_{W^+ (k) [u]} M$ is
  contractible. In this case, we apply the exact functor $- \otimes_{W^+ (k)
  [u]} M$ to the cofiber sequence
  \begin{equation}
    W^+ (k) [u] \xrightarrow{m_u} W^+ (k) [u] \rightarrow W^+ (k)
    \label{eq:cofiber-poly}
  \end{equation}
  indicated in Proposition~\ref{prop:cofiber-mult} obtaining that the
  base-changed map $M \xrightarrow{m_u \otimes_{W^+ (k) [u]} M} M$ is an
  equivalence of spectra. Note that this map is just the multiplication map,
  denoted by $m_{M, u}$. Now we look at Proposition~\ref{prop:calc-compl}: the
  $(u)$-completion of the $W^+ (k) [u]$-module $M$ is the cofiber of the
  canonical map $T (M) \rightarrow M$, where $T (M)$ is the limit of the tower
  \[ \cdots \xrightarrow{m_{M, u}} M \xrightarrow{m_{M, u}} M
     \xrightarrow{m_{M, u}} M \]
  Since all maps in the tower are equivalences of spectra, we deduce that the
  canonical map $T (M) \rightarrow M$ is an equivalence of spectra, which
  implies that the $(u)$-completion of the $W^+ (k) [u]$-module $M$ is
  contractible. In particular, the $W^+ (k) [u]$-module $M$ is assumed to be
  $(u)$-complete, therefore the spectrum $M$ is contractible.
  
  If, on the other hand, $W^+ (k) \otimes_{W^+ (k) [[u]]} M$ is contractible,
  then to adopt the proof above, it suffices to establish the cofiber sequence
  \begin{equation}
    W^+ (k) [[u]] \xrightarrow{m_u} W^+ (k) [[u]] \rightarrow W^+ (k)
    \label{eq:cofiber-power-series}
  \end{equation}
  We apply the $(u)$-complete functor to the cofiber sequence
  \eqref{eq:cofiber-poly}, and note that the $W^+ (k) [u]$-module $W^+ (k)$ is
  $(u)$-nilpotent (in fact, multiplying $u$ is the zero map on $W^+ (k)$),
  therefore $W^+ (k)$ is $(u)$-complete by Corollary~\ref{cor:nil-compl},
  which leads to the cofiber sequence \eqref{eq:cofiber-power-series}. The
  rest of the proof is same as before.
\end{proof}

\subsection{The Breuil-Kisin case}

As before, we fix a complete DVR $(A, \mathfrak{m})$ of mixed characteristics
$(0, p)$ with residue field $k$ being perfect, absolute ramification index
$e$, a uniformizer $\varpi \in \mathfrak{m}$ and an Eisenstein $W
(k)$-polynomial $E (u) \in W (k) [u]$ which induces an isomorphism $W (k) [u]
/ (E (u)) \xrightarrow{\sim} A, u \mapsto \varpi$ as in
Proposition~\ref{prop:struct-thm-cdvr-mixed}. As in Remark~\ref{cons:f_p} and
Remark~\ref{cons:f_R}, $1 - E (u) \in W (k) [[u]] = \pi_1 (\tmop{BGL}_1 (W^+
(k) [[u]]))$ gives rise to a map $f_E \of \Omega^2 S^3 \rightarrow
\tmop{BGL}_1 (W^+ (k) [[u]])$. The proof of Lemma~\ref{lem:pi0Mf} results in
the following analogue:

\begin{lemma}
  The zeroth homotopy group of the $\mathbb{E}_2$-Thom spectrum $M f_E$
  associated to the map $f_E$ is isomorphic to the $W (k)$-algebra $W (k)
  [[u]] / (E (u)) \cong W (k) [u] / (E (u)) \cong A$.
\end{lemma}

The $W (k) [u]$-module structure on $A$ gives rise to a $W^+ (k) [u]$-module
structure on $H A$. Since $A$ is $\mathfrak{m}= (\varpi)$-adically complete,
the $W (k) [u]$-module structure on $A$ also gives rise to a $W (k)
[[u]]$-module structure on $A$ and consequently a $W^+ (k) [[u]]$-module
structure on $H A$. We readily [details needed] check that these structures
are compatible, in the sense that the $W^+ (k) [u]$-module structure on $H A$
coincides with the image of the $W^+ (k) [[u]]$-module $H A$ under the
forgetful functor $\tmop{Mod}_{W^+ (k) [[u]]} \rightarrow \tmop{Mod}_{W^+ (k)
[u]}$. Matthew Morrow proposed the following analogue of the Hopkins-Mahowald
theorem:

\begin{theorem}
  \label{thm:cdvr-thom}The truncation map $t_E \of M f_E \rightarrow H \pi_0
  (M f_E) \cong H A$ of $\mathbb{E}_2$-$W^+ (k) [[u]]$-algebras is an
  equivalence of spectrum. Thus the Eilenberg-Maclane spectrum $H A$ is the
  $\mathbb{E}_2$-Thom spectrum $M f_E$ associated to the map $f_E \of \Omega^2
  S^3 \rightarrow \tmop{BGL}_1 (W^+ (k) [[u]])$.
\end{theorem}

\begin{corollary}[see {\cite[Remark~3.4]{Krause2019}}]
  \label{cor:HA-base-change-free}The $\mathbb{E}_2$-$H A$-algebra $H A
  \otimes_{W^+ (k) [[u]]} H A$ is a free $\mathbb{E}_2$-$H A$-algebra on a
  single generator in degree 1.
\end{corollary}

\begin{proof}
  The strategy is already covered in the proof of
  Lemma~\ref{lem:perf-thom-special} and Lemma~\ref{lem:eq-thom-base-change}.
  Since this pattern will appear again soon, we find it beneficial to present
  again. Let's recall that the $\mathbb{E}_2$-Thom spectrum $M f_E$ is the
  colimit of the composite functor
  \[ \Omega^2 S^3 \xrightarrow{f_E} \tmop{BGL}_1 (W^+ (k) [[u]]) \rightarrow
     \tmop{Mod}_{W^+ (k) [[u]]} \]
  which by abuse of notation will be still denoted by $f_E$.
  
  Since the base change functor $H A \otimes_{W^+ (k) [[u]]} - \of
  \tmop{Mod}_{W^+ (k) [[u]]} \rightarrow \tmop{Mod}_{H A}$ is a left adjoint,
  it commutes with colimits, we deduce that $H A \otimes_{W^+ (k) [[u]]} M f_E
  \simeq M (f_E \otimes_{W^+ (k) [[u]]} H A)$, where $f_E \otimes_{W^+ (k)
  [[u]]} H A$ is the map $\Omega^2 S^3 \rightarrow \tmop{BGL}_1 (H A)$.
  
  As in the proof of Lemma~\ref{lem:perf-thom-special}, we can identify map as
  follows: we pick the image of $1 - E (u) \in \tmop{GL}_1 (W (k) [[u]])$
  under the map $\tmop{GL}_1 (W (k) [[u]]) \rightarrow \tmop{GL}_1 (A)$, that
  is, the element $1 \in \tmop{GL}_1 (A) \cong \pi_1 (\tmop{BGL}_1 (H A))$,
  which gives rise to the constant map $S^1 \rightarrow \tmop{BGL}_1 (H A)$
  and consequently the constant map $f_A \of \Omega^2 S^3 \rightarrow
  \tmop{BGL}_1 (H A)$, as in Remark~\ref{cons:f_p} and Remark~\ref{cons:f_R}.
  
  In conclusion, the map $f_E \otimes_{W^+ (k)} H A \of \Omega^2 S^3
  \rightarrow \tmop{BGL}_1 (H A)$ coincides with the constant map $f_A$, and
  the $\mathbb{E}_2$-Thom spectrum $M f_A$ is thus the colimit of a constant
  map, which evaluates to $H A \otimes \Omega^2 S^3$, the free
  $\mathbb{E}_2$-$H A$-algebra on a single generator in degree 1.
\end{proof}

Recall that $E (u) \in W (k) [u]$ is an Eisenstein $W (k)$-polynomial. Let
$a_0$ denote the constant term of $E (u)$. By assumption, $p \divides a_0$ but
$p^2 \ndivides a_0$. Let $a_0 = pb_0$ where $b_0 \in W (k)$. Since $p$ is not
a zero-divisor in $W (k)$, we have $p \ndivides b_0$, which implies that the
image of $b_0$ in $W (k) / p \cong k$ is invertible since $k$ is a field. Now
since $W (k)$ is $p$-adically complete, we have $b_0 \in \tmop{GL}_1 (W (k))$.

The strategy to prove Theorem~\ref{thm:cdvr-thom} is similar to the approach
to attack Theorem~\ref{thm:main}. We first show that the base change of the
truncation map $t_E$ along the map $W^+ (k) [[u]] \rightarrow W^+ (k)$
coincides with the truncation map $t_{k, a_0}$, then it follows from
Lemma~\ref{lem:eq-postnikov-perf} that the base changed map $W^+ (k)
\otimes_{W^+ (k) [[u]]} t_E \simeq t_{k, a_0}$ is an equivalence of spectra,
and by completeness, we deduce that the map $t_E$ is also an equivalence of
spectra by Lemma~\ref{lem:compl-mod-u}.

\begin{lemma}
  \label{lem:eq-thom-base-change-bk}There is a canonical equivalence $M f_{k,
  a_0} \xrightarrow{\simeq} W^+ (k) \otimes_{W^+ (k) [[u]]} M f_E$ of $W^+
  (k)$-modules.
\end{lemma}

\begin{proof}
  We will duplicate the proof of Lemma~\ref{lem:eq-thom-base-change}. The
  image of the multiplication map $m_{1 - E (u)} \of W^+ (k) [[u]] \rightarrow
  W^+ (k) [[u]]$ under the base change functor $W^+ (k) \otimes_{W^+ (k)
  [[u]]} - \of \tmop{Mod}_{W^+ (k) [[u]]} \rightarrow \tmop{Mod}_{W^+ (k)}$ is
  the multiplication map $m_{1 - a_0} \of W^+ (k) \rightarrow W^+ (k)$. Note
  also that the base change functor is symmetric monoidal. Now we conclude
  that the map $f_{k, a_0}$ coincides with the composite map
  \[ \Omega^2 S^3 \xrightarrow{f_E} \tmop{BGL}_1 (W^+ (k) [[u]])
     \xrightarrow{W^+ (k) \otimes_{W^+ (k) [[u]]} -} \tmop{BGL}_1 (W^+ (k)) \]
  Thus by commuting the colimit and the base-change, we obtain
  \begin{eqnarray*}
    M f_{k, a_0} & = & \tmop{colim} (W^+ (k) \otimes_{W^+ (k) [[u]]} f_E)\\
    & \xrightarrow{\simeq} & W^+ (k) \otimes_{W^+ (k) [[u]]} \tmop{colim}
    f_E\\
    & = & W^+ (k) \otimes_{W^+ (k) [[u]]} M f_E
  \end{eqnarray*}
  where by abuse of notation, the colimit of the maps $f_{k, a_0}$ (or $f_E$
  respectively) are understood as the colimit of the maps $f_{k, a_0}$ (or
  $f_E$ respectively) composed with the functor $\tmop{BGL}_1 (W^+ (k))
  \rightarrow \tmop{Mod}_{W^+ (k)}$ (or $\tmop{BGL}_1 (W^+ (k) [[u]])
  \rightarrow \tmop{Mod}_{W^+ (k) [[u]]}$ respectively) as in the definition
  of Thom spectra.
\end{proof}

\begin{lemma}
  \label{lem:eq-HA-base-change}There is a canonical equivalence $W^+ (k)
  \otimes_{W^+ (k) [[u]]} H A \xrightarrow{\simeq} H k$ of $W^+ (k)$-modules.
\end{lemma}

\begin{proof}
  As in the proof of Lemma~\ref{lem:compl-mod-u}, we identify $W^+ (k)$ with
  the cofiber of the multiplication map $m_u \of W^+ (k) [[u]] \rightarrow W^+
  (k) [[u]]$ which gives us an equivalence
  \[ W^+ (k) \otimes_{W^+ (k) [[u]]} H A \simeq \tmop{cofib} \left( H A
     \xrightarrow{m_{H A, u}} H A \right) \]
  Now by the definition of the $W^+ (k) [[u]]$-module structure on $H A$ and
  that $u$ is not a zero-divisor in $A$, we have the equivalence $\tmop{cofib}
  \left( H A \xrightarrow{m_{H A, u}} H A \right) \simeq H \left( \tmop{coker}
  \left( A \xrightarrow{\varpi} A \right) \right) \simeq H k$. Thus we obtain
  an equivalence $W^+ (k) \otimes_{W^+ (k) [[u]]} H A \simeq H k$. We can
  readily check [details needed] that this equivalence could be described as
  follows: consider the commutative diagram in the $\infty$-category of
  $\mathbb{E}_{\infty}$-rings
  \[ \begin{array}{ccc}
       W^+ (k) [[u]] & \longrightarrow & W^+ (k)\\
       \longdownarrow &  & \longdownarrow\\
       H A & \longrightarrow & H k
     \end{array} \]
  where the left vertical map is the composite map $W^+ (k) [[u]] \rightarrow
  H (\pi_0 (W^+ (k) [[u]])) \simeq H (W (k) [[u]]) \xrightarrow{u \mapsto
  \varpi} H A$ (where the first map is the Postnikov section). The commutative
  diagram induces a map $W^+ (k) \otimes_{W^+ (k) [[u]]} H A \rightarrow H k$
  (note that the left hand side is a pushout of $\mathbb{E}_{\infty}$-rings),
  which coincides with the equivalence obtained above.
\end{proof}

\begin{lemma}
  The equivalences in Lemma~\ref{lem:eq-thom-base-change-bk} and
  Lemma~\ref{lem:eq-HA-base-change} assembles into a commutative diagram:
  \[ \begin{array}{ccc}
       M f_{k, a_0} & \xrightarrow{t_{k, a_0}} & H k\\
       \longdownarrow \simeq &  & \longuparrow \simeq\\
       W^+ (k) \otimes_{W^+ (k) [[u]]} M f_E & \longrightarrow & W^+ (k)
       \otimes_{W^+ (k) [[u]]} H A
     \end{array} \]
  where the top horizontal map is the 0th Postnikov section $t_{k, a_0}$
  defined in Proposition~\ref{prop:main} and the bottom horizontal map is the
  base-changed 0th Postnikov section $W^+ (k) \otimes_{W^+ (k) [[u]]} t_E$.
\end{lemma}

\begin{proof}
  As in the proof of Lemma~\ref{lem:eq-base-change-sphwitt-kappa}, it suffices
  to show that the composite map on the 0th homotopy group $\pi_0 (M f_{k,
  a_0}) \rightarrow \pi_0 (W^+ (k) \otimes_{W^+ (k) [[u]]} M f_E) \rightarrow
  \pi_0 (W^+ (k) \otimes_{W^+ (k) [[u]]} H A) \rightarrow \pi_0 (H k) \cong k$
  is an isomorphism, which follows from an explicit element chasing.
\end{proof}

Combined with Lemma~\ref{lem:eq-postnikov-perf}, we obtain that

\begin{corollary}
  The base-changed map $W^+ (k) \otimes_{W^+ (k) [[u]]} t_E \of W^+ (k)
  \otimes_{W^+ (k) [[u]]} M f_E \rightarrow W^+ (k) \otimes_{W^+ (k) [[u]]} H
  A$ is an equivalence of $W^+ (k)$-modules.
\end{corollary}

Apply Lemma~\ref{lem:compl-mod-u} to the cofiber $\tmop{cofib} (t_E)$, we
deduce that

\begin{corollary}
  \label{cor:eq-tr-u-compl}The map $t_E \of M f_E \rightarrow H A$ is an
  equivalence of spectra after $(u)$-completion.
\end{corollary}

As in Lemma~\ref{lem:HR-p-compl}, we deduce from Theorem~\ref{thm:cplhomot}
that

\begin{lemma}
  The $W^+ (k) [[u]]$-module $H A$ is $(u)$-complete.
\end{lemma}

Now, given the nontrivial topological input
Proposition~\ref{prop:finite-Kan-double-loop}, as in
Lemma~\ref{lem:thom-p-compl} and Corollary~\ref{cor:cofib-tr-p-compl}, we
deduce that

\begin{lemma}
  The $W^+ (k) [[u]]$-module $M f_E$ is $(u)$-complete.
\end{lemma}

\begin{corollary}
  The cofiber $\tmop{cofib} (t_E)$ is a $(u)$-complete $W^+ (k) [[u]]$-module,
  and thus the map $t_E$ is an equivalence of spectra by
  Corollary~\ref{cor:eq-tr-u-compl}.
\end{corollary}

This completes the proof of Theorem~\ref{thm:cdvr-thom}.

\subsection{Complete regular local rings}

Inspired by {\cite[Section~9]{Krause2019}}, we will provide a Hopkins-Mahowald
theorem for complete regular local rings of mixed characteristic. We will show
how to modify our proof of Theorem~\ref{thm:cdvr-thom} to deduce this. Note
that this is also a special case of Question~\ref{ques:prism-thom}, by
{\cite[Remark~3.11]{Bhatt2019}}.

We need some preparations in higher algebra:

Let $W^+ (k) [u_1, \ldots, u_n]$ be the ``$n$-variate polynomial $W^+
(k)$-algebra'', that is, the $\mathbb{E}_{\infty}$-$W^+ (k)$-algebra $W^+ (k)
\otimes_{\mathbb{S}} \mathbb{S} [\mathbb{N}^n]$. Since the space
$\mathbb{N}^n$ is endowed with discrete topology, parallel to
Proposition~\ref{prop:single-var-poly-homotopy-grps}, we have

\begin{proposition}
  \label{prop:mult-var-poly-homotopy-grps}As a $W^+ (k)$-module, $W^+ (k)
  [u_1, \ldots, u_n]$ is equivalent to the direct sum $\bigoplus_{\alpha \in
  \mathbb{N}^n} u^{\alpha} W^+ (k)$, a free $W^+ (k)$-module. The graded
  homotopy group $\pi_{\ast} (W^+ (k) [u_1, \ldots, u_n])$, as a
  (graded-commutative) $\pi_{\ast} (W^+ (k))$-algebra, is equivalent to
  $\pi_{\ast} (W^+ (k)) [u_1, \ldots, u_n]$, where $\deg u_1 = \cdots = \deg
  u_n = 0$.
\end{proposition}

Now let $W^+ (k) [[u_1, \ldots, u_n]]$ be the $(u_1, \ldots, u_n)$-completion
of the $\mathbb{E}_{\infty}$-$W^+ (k)$-algebra $W^+ (k) [u_1, \ldots, u_n]$.
By induction on $n \in \mathbb{N}_{> 0}$ and argue as in
Proposition~\ref{prop:single-var-power-series-homotopy-grps} [details needed],
we obtain:

\begin{proposition}
  \label{prop:mult-var-power-series-homotopy-grps}The
  $\mathbb{E}_{\infty}$-$W^+ (k)$-algebra $W^+ (k) [[u_1, \ldots, u_n]]$ is
  connective. The zeroth homotopy group of $\pi_0 (W^+ (k) [[u_1, \ldots,
  u_n]])$ is isomorphic to the $(u_1, \ldots, u_n)$-adic completion of the
  polynomial $W (k)$-algebra $W (k) [u_1, \ldots, u_n]$, that is, the formal
  power series $W (k)$-algebra $W (k) [[u_1, \ldots, u_n]]$, as $W
  (k)$-algebras.
\end{proposition}

Similarly, argue inductively on $n \in \mathbb{N}_{> 0}$ as in
Lemma~\ref{lem:compl-mod-u}, we obtain:

\begin{lemma}
  Let $M$ be a $W^+ (k) [u_1, \ldots, u_n]$-(or $W^+ (k) [[u_1, \ldots,
  u_n]]$-)module (spectrum). If the spectrum $W^+ (k) \otimes_{W^+ (k) [u_1,
  \ldots, u_n]} M$ (or $W^+ (k) \otimes_{W^+ (k) [[u_1, \ldots, u_n]]} M$
  respectively) is contractible, then so is the $(u_1, \ldots,
  u_n)$-completion of the spectrum $M$. In particular, if furthermore $W^+ (k)
  [u_1, \ldots, u_n]$-(or $W^+ (k) [[u_1, \ldots, u_n]]$-)module $M$ is
  assumed to be $(u_1, \ldots, u_n)$-complete, then the spectrum $M$ is
  contractible.
\end{lemma}

We note that in these inductive arguments, we heavily depend on
Proposition~\ref{prop:compl-gens}.

\

Now we are ready to formulate the Hopkins-Mahowald theorem for complete
regular local rings. We fix a positive integer $n \in \mathbb{N}_{> 0}$, a
perfectoid ring $R$. As in Section~\ref{sec:main}, let $\theta \of W
(R^{\flat}) \rightarrow R$ be Fontaine's pro-infinitesimal thickening. Let
$\phi \in W (R^{\flat}) [[u_1, \ldots, u_n]]$ be formal power series such that
$\phi (0, \ldots, 0) \in W (R^{\flat})$ is a generator of $\ker \theta$. We
recall that $\ker \theta$ is principal by definition. We note that the element
$1 - \phi (u_1, \ldots, u_n) \in W (R^{\flat}) [[u_1, \ldots, u_n]]$ is
invertible, since $1 - \phi (0, \ldots, 0) \in W (R^{\flat})$ is invertible as
the ring $W (R^{\flat})$ is $\ker \theta$-adically complete. As in
Remark~\ref{cons:f_p} and Remark~\ref{cons:f_R}, the element $1 - \phi (u_1,
\ldots, u_n) \in \tmop{GL}_1 (W (R^{\flat}) [[u_1, \ldots, u_n]])$ gives rise
to an $\mathbb{E}_2$-map $f \of \Omega^2 S^3 \rightarrow \tmop{BGL}_1 (W
(R^{\flat}) [[u_1, \ldots, u_n]])$. The proof of Lemma~\ref{lem:pi0Mf} results
in the following analogue:

\begin{lemma}
  The zeroth homotopy group of the $\mathbb{E}_2$-Thom spectrum $M f$
  associated to the map $f$ is isomorphic to the $W (R^{\flat})$-algebra $W
  (R^{\flat}) [[u_1, \ldots, u_n]] / (\phi (u_1, \ldots, u_n))$.
\end{lemma}

We now phrase the following variant of the Hopkins-Mahowald theorem:

\begin{theorem}
  \label{thm:perfd-regular-thom}The truncation map $t \of M f \rightarrow H
  \pi_0 (M f) \cong H W (R^{\flat}) [[u_1, \ldots, u_n]] / (\phi (u_1, \ldots,
  u_n))$ of $\mathbb{E}_2$-$W^+ (R^{\flat}) [[u_1, \ldots, u_n]]$-algebras is
  an equivalence of spectrum. Thus the Eilenberg-Maclane spectrum $H W
  (R^{\flat}) [[u_1, \ldots, u_n]] / (\phi (u_1, \ldots, u_n))$ is the
  $\mathbb{E}_2$-Thom spectrum $M f$ associated to the map $f \of \Omega^2 S^3
  \rightarrow \tmop{BGL}_1 (W^+ (R^{\flat}) [[u_1, \ldots, u_n]])$.
\end{theorem}

The proof is parallel to that of Theorem~\ref{thm:cdvr-thom}, which we will
omit. Now let $(A, \mathfrak{m})$ be a complete regular local ring with
residue field $k = A /\mathfrak{m}$ being perfect of characteristic $p$. We
also assume that $p \neq 0$ in $A$. Let $(a_1, \ldots, a_n) \subseteq
\mathfrak{m}$ be a regular sequence which generates the maximal ideal
$\mathfrak{m}$. We need the following lemma:

\begin{lemma}[{\cite[Lemma~9.2]{Krause2019}}]
  \label{lem:struct-thm}There exists a map $W (k) [[u_1, \ldots, u_n]]
  \rightarrow A$ of rings given by $u_i \mapsto a_i$ for $i = 1, \ldots, n$,
  which is surjective with kernel being principal, generated by a formal power
  series $\phi \in W (k) [[u_1, \ldots, u_n]]$ with $\phi (0, \ldots, 0) = p$.
\end{lemma}

\begin{proof}
  First, the isomorphism $k \rightarrow A /\mathfrak{m}$ lifts to a map $W (k)
  \rightarrow A$ since $A$ is $\mathfrak{m}$-adically complete, see
  Example~\ref{ex:witt} or {\cite[Section~II.5, Proposition~10]{Serre1979}}.
  The map $W (k) [[u_1, \ldots, u_n]] \rightarrow A$ is then well-defined
  since $A$ is $\mathfrak{m}$-adic complete. Let $C, K$ be the cokernel and
  the kernel of the map $W (k) [[u_1, \ldots, u_n]] \rightarrow A$ of $W (k)
  [[u_1, \ldots, u_n]]$-modules. By right-exactness of classical tensor
  products, we have
  \[ \tmop{Tor}_0^{W (k) [[u_1, \ldots, u_n]]} (C, W (k)) \cong \tmop{coker}
     (W (k) \rightarrow k) \cong 0 \]
  Now by inspecting the exact sequence $0 \rightarrow \mathfrak{m} \rightarrow
  A \rightarrow k \rightarrow 0$, we deduce that $A$ is a finitely generated
  $W (k) [[u_1, \ldots, u_n]]$-module, therefore so is $C$. We deduce from
  Nakayama's lemma that $C \cong 0$, therefore the map $W (k) [[u_1, \ldots,
  u_n]]$ is surjective. Now we obtain a short exact sequence of $W (k) [[u_1,
  \ldots, u_n]]$-modules
  \[ 0 \rightarrow K \rightarrow W (k) [[u_1, \ldots, u_n]] \rightarrow A
     \rightarrow 0 \]
  which gives rise to an exact sequence of $W (k)$-modules
  \[ \tmop{Tor}_1^{W (k) [[u_1, \ldots, u_n]]} (A, W (k)) \rightarrow
     \tmop{Tor}_0^{W (k) [[u_1, \ldots, u_n]]} (K, W (k)) \rightarrow W (k)
     \rightarrow k \rightarrow 0 \]
  Since $(a_1, \ldots, a_n)$ is a regular sequence, it is also Koszul regular
  {\cite[\href{}{Tag 062F}]{stacks-project}}, hence $\tmop{Tor}_1^{W (k)
  [[u_1, \ldots, u_n]]} (A, W (k)) \cong 0$. Thus
  \[ \tmop{Tor}_0^{W (k) [[u_1, \ldots, u_n]]} (K, W (k)) \cong \ker (W (k)
     \rightarrow k) \cong pW (k) \]
  We pick a lift $\phi \in K$ of $p \in pW (k)$. By Nakayama's lemma, the $W
  (k) [[u_1, \ldots, u_n]]$ module (and hence the ideal) $K$ is generated by
  the element $\phi \in K$. Furthermore, by multiplying an invertible element
  in $W (k)$, we can assume that the lift $\phi$ is so chosen that $\phi (0,
  \ldots, 0) = p$.
\end{proof}

\begin{remark}
  Our proof of Lemma~\ref{lem:struct-thm} leads to a more general result: Let
  $A$ be a commutative ring with an ideal $I \subseteq A$ generated by a
  (Koszul) regular sequence $(a_1, \ldots, a_n) \subseteq I$. If $A$ is both
  $p$-adically complete and $I$-adically complete, and $R \assign A / I$ is a
  perfectoid ring, then by Proposition~\ref{prop:Fontaine-inf-thickening},
  there exists a unique map $W (R^{\flat}) \rightarrow A$ such that the
  composite map $W (R^{\flat}) \rightarrow A \rightarrow R$ coincides with
  Fontaine's map, which allows us to view $A$ as a $W (R^{\flat})$-algebra.
  Now we consider the map $\varphi \of W (R^{\flat}) [[u_1, \ldots, u_n]]
  \rightarrow A$ of $W (R^{\flat})$-algebras given by $u_i \mapsto a_i$ for $i
  = 1, \ldots, n$. Our proof of Lemma~\ref{lem:struct-thm} implies that the
  map $\varphi$ is surjective with kernel being principal, generated by a
  formal power series $\phi \in W (R^{\flat}) [[u_1, \ldots, u_n]]$ such that
  $\phi (0, \ldots, 0)$ generates the kernel $\ker (\theta)$ of Fontaine's map
  $\theta \of W (R^{\flat}) \rightarrow R$.
\end{remark}

\begin{corollary}
  Let $\phi \in W (k) [[u_1, \ldots, u_n]]$ be a power series as described in
  Lemma~\ref{lem:struct-thm}. Let $f \of \Omega^2 S^3 \rightarrow \tmop{BGL}_1
  (W^+ (k) [[u_1, \ldots, u_n]])$ be the map given by the element $1 - \phi
  (u_1, \ldots, u_n) \in \tmop{GL}_1 (W (k) [[u_1, \ldots, u_n]])$. Then the
  $\mathbb{E}_2$-Thom spectrum $M f$ associated to the map $f$ is as an
  $\mathbb{E}_2$-$W^+ (k) [[u_1, \ldots, u_n]]$-algebra equivalent to the
  Eilenberg-Maclane spectrum $H A$ of the complete regular local ring $A$ (of
  mixed characteristic).
\end{corollary}

\begin{proof}
  It follows from Theorem~\ref{thm:perfd-regular-thom} by taking $R = k$ and
  Lemma~\ref{lem:struct-thm}.
\end{proof}

\section{Characterizing Thom spectra as quotients of free
$\mathbb{E}_2$-algebras}\label{sec:thom-free-E2-algs}

In this section, we will discuss an alternative characterization of Thom
spectra which we learn from {\cite{Antolin-Camarena2014}}. This
characterization will enable us to peel off some redundant restraints in the
definition of Thom spectra. We will rephrase Question~\ref{ques:prism-thom}
more broadly, and give a toy example related to the Breuil-Kisin case. We note
that in fact, we have already used this characterization in
Lemma~\ref{lem:thom-universal}.

We first present a theorem which we learn from Antolín-Camarena and Barthel's
paper:

\begin{remark}
  Let $R$ be an $\mathbb{E}_{\infty}$-ring. Let $R [\Omega^2 S^2]$ be the free
  $\mathbb{E}_2$-$R$-algebra on a single generator in degree 0. Then for all
  $\mathbb{E}_2$-$R$-algebra $S$ and elements $x \in \pi_0 (S)$, the universal
  property of free $\mathbb{E}_2$-$R$-algebras gives rise to a map $R
  [\Omega^2 S^2] \rightarrow S$ which maps the generator (in fact, a connected
  component) to $x$. We will call this map the evaluation map of $R [\Omega^2
  S^2]$ at $x$.
\end{remark}

\begin{theorem}[{\cite[Theorem~4.10]{Antolin-Camarena2014}}]
  \label{thm:versal-alg}Let $R$ be an $\mathbb{E}_{\infty}$-ring and $\alpha
  \in \pi_1 (\tmop{BGL}_1 (R)) \cong \tmop{GL}_1 (\pi_0 R)$. Let $q \of S^1
  \rightarrow \tmop{BGL}_1 (R)$ a loop representing $\alpha \in \pi_1
  (\tmop{BGL}_1 (R))$. Let $f \of \Omega^2 S^3 \rightarrow \tmop{BGL}_1 (R)$
  be the double loop map associated to $q$ (see Remark~\ref{rem:bgl1mongrpd}).
  Then the $\mathbb{E}_2$-Thom spectrum $M f$ associated to the
  $\mathbb{E}_2$-map $f$ fits into a pushout diagram of
  $\mathbb{E}_2$-$R$-algebras:
  \[ \begin{array}{ccc}
       R [\Omega^2 S^2] & \longrightarrow & R\\
       \longdownarrow &  & \longdownarrow\\
       R & \longrightarrow & M f
     \end{array} \]
  where $R [\Omega^2 S^2] \cong R \otimes_{\mathbb{S}} \Sigma_+^{\infty} S^2$
  is the free $\mathbb{E}_2$-$R$-algebra on a single generator in degree 0,
  and two maps $R [\Omega^2 S^2] \rightarrow R$ are evaluation maps of $R
  [\Omega^2 S^3]$ at $0 \in \pi_0 R$ and $1 - \alpha \in \pi_0 R$
  respectively.
\end{theorem}

\begin{remark}
  Theorem~\ref{thm:versal-alg} shows that the Thom spectrum description is
  equivalent to the pushout-diagram description. However, we note that the
  pushout-diagram description is more general in the sense that even if
  $\alpha \in \pi_0 R$ is not invertible, the pushout-diagram description is
  still valid while we can no longer, at least superficially, give a Thom
  spectrum description. We find it easier to write down proofs for Thom
  spectrum description so we adapted the Thom spectrum description for
  perfectoid rings.
\end{remark}

We can now rephrase Question~\ref{ques:prism-thom} as follows:

\begin{question}
  \label{ques:prism-E2}Given an orientable prism $(A, I = (d))$ . When can we
  find an $\mathbb{E}_{\infty}$-ring spectrum $A^+$ (which satisfies some
  hypotheses related to $A$. A naive guess would be that $\pi_0 (A^+) = A$) so
  that the Eilenberg-Maclane spectrum $H (A / I)$ as an
  $\mathbb{E}_2$-$A^+$-algebra fits into a pushout diagram
  \[ \begin{array}{ccc}
       A^+ [\Omega^2 S^2] & \longrightarrow & A^+\\
       \longdownarrow &  & \longdownarrow\\
       A^+ & \longrightarrow & H (A / I)
     \end{array} \]
  such that two maps $A^+ [\Omega^2 S^2] \rightarrow A^+$ are evaluation maps
  of the free $\mathbb{E}_2$-$A^+$-algebra $A^+ [\Omega^2 S^2]$ at $0 \in
  \pi_0 (A^+)$ and $d \in \pi_0 (A^+)$ respectively.
\end{question}

\begin{remark}
  Theorem~\ref{thm:versal-alg} shows that Theorem~\ref{thm:main} answers this
  question affirmatively when $(A, I)$ is a perfect prism $(W (R^{\flat}),
  \ker \theta)$, with $A^+ \assign W^+ (R^{\flat})$.
\end{remark}

\begin{remark}
  Similarly, Theorem~\ref{thm:cdvr-thom} answers this question affirmatively
  when $(A, I)$ is a prism $(W (k) [[u]], (E (u)))$ associated to Breuil-Kisin
  cohomology where $k$ is a perfect $\mathbb{F}_p$-algebra and $E (u) \in W
  (k) [u]$ is an Eisenstein polynomial.
\end{remark}

We now announce a toy example of a variant of Theorem~\ref{thm:cdvr-thom}. As
there, we fix a complete DVR $(A, \mathfrak{m})$ of mixed characteristics $(0,
p)$ with residue field $k$ being perfect, absolute ramification index $e$, a
uniformizer $\varpi \in \mathfrak{m}$ and an Eisenstein $W (k)$-polynomial $E
(u) \in W (k) [u]$ which induces an isomorphism $W (k) [u] / (E (u))
\xrightarrow{\sim} A, u \mapsto \varpi$ as in
Proposition~\ref{prop:struct-thm-cdvr-mixed}.

\begin{theorem}
  \label{thm:cdvr-E2}The $(u)$-completion of the total cofiber of the
  commutative diagram of $\mathbb{E}_2$-$W^+ (k) [u]$-algebras
  \[ \begin{array}{ccc}
       W^+ (k) [u] \otimes_{\mathbb{S}} \mathbb{S} [\Omega^2 S^2] &
       \longrightarrow & W^+ (k) [u]\\
       \longdownarrow &  & \longdownarrow\\
       W^+ (k) [u] & \longrightarrow & H A
     \end{array} \]
  is contractible, where two maps $W^+ (k) [u] \otimes_{\mathbb{S}} \mathbb{S}
  [\Omega^2 S^3] \rightarrow W^+ (k) [u]$ are given by evaluation maps at $0
  \in \pi_0 (W^+ (k) [u])$ and $E (u) \in \pi_0 (W^+ (k) [u])$ respectively.
  Equivalently put, the commutative diagram above induces an equivalence of
  $W^+ (k) [u]$-modules from the $\mathbb{E}_2$-pushout of the diagram $W^+
  (k) [u] \leftarrow W^+ (k) [u] \otimes_{\mathbb{S}} \mathbb{S} [\Omega^2
  S^2] \rightarrow W^+ (k) [u]$ to the Eilenberg-Maclane spectrum $H A$ after
  $(u)$-completion.
\end{theorem}

\begin{corollary}[{\cite[Remark~3.4]{Krause2019}}]
  The $\mathbb{E}_2$-$H A$-algebra $H A \otimes_{W^+ (k) [u]} H A$ is the
  $(p)$-completion of the free $\mathbb{E}_2$-$H A$-algebra on a single
  generator in degree 1.
\end{corollary}

\begin{proof}
  Note that $E (u)$ vanishes after tensoring $H A$, and that $(u)$-completion
  coincides with $(p)$-completion for $H A$ since $\varpi^e / p$ is an
  invertible element, the result follows.
\end{proof}

We sketch a proof of Theorem~\ref{thm:cdvr-E2}, which is totally parallel to
that of Theorem~\ref{thm:cdvr-thom}.

\begin{proof*}{A sketch of a proof of Theorem~\ref{thm:cdvr-E2}}
  Let $X$ be the pushout of the diagram $W^+ (k) [u] \leftarrow W^+ (k) [u]
  \otimes_{\mathbb{S}} \mathbb{S} [\Omega^2 S^2] \rightarrow W^+ (k) [u]$ in
  question. We first check that the induced map $X \rightarrow H A$ is the 0th
  Postnikov section. Then we perform a base change $W^+ (k) \otimes_{W^+ (k)
  [u]} -$. We show that after such a base change, the induced map $X
  \rightarrow H A$ becomes an equivalence, given by
  Theorem~\ref{thm:versal-alg} and Lemma~\ref{lem:eq-postnikov-perf}. We then
  conclude the result by Corollary~\ref{cor:eq-tr-u-compl}.
\end{proof*}

\appendix\section{Recollection of Higher Algebra}

This appendix is devoted to a recollection of basic facts in Higher Algebra
needed in the main text. Our main reference is {\cite{Lurie2017}},
{\cite{Lurie2018}} and {\cite{Lurie2018a}}.

\subsection{Finiteness properties of rings and modules}

We will include some definitions and properties from
{\cite[Section~7.2.4]{Lurie2017}}.

\begin{definition}[{\cite[Notation~7.1.1.10,
Proposition~7.1.1.13]{Lurie2017}}]
  Given a connective $\mathbb{E}_1$-ring $R$, there is a canonical accessible
  $t$-structure on $\tmop{LMod}_R$ determined by subcategories
  $(\tmop{LMod}_R)_{\geq 0}$ and $(\tmop{LMod}_R)_{\leq 0}$, where
  $(\tmop{LMod}_R)_{\geq 0}$ is the full subcategory of $\tmop{LMod}_R$
  spanned by those left $R$-modules $M$ for which $\pi_n M \cong 0$ for $n <
  0$, and $(\tmop{LMod}_R)_{\leq 0}$ is the full subcategory of
  $\tmop{LMod}_R$ spanned by those left $R$-modules $M$ for which $\pi_n M
  \cong 0$ for $n > 0$.
\end{definition}

\begin{proposition}[{\cite[Proposition~7.1.1.13]{Lurie2017}}]
  Let $R$ be a connective $\mathbb{E}_1$-ring, then the subcategories
  $(\tmop{LMod}_R)_{\geq 0}, (\tmop{LMod}_R)_{\leq 0} \subseteq \tmop{LMod}_R$
  are stable under small products and small filtered colimits.
\end{proposition}

\begin{definition}[{\cite[Proposition~7.2.2.10]{Lurie2017}}]
  Let $M$ be a left module over an $\mathbb{E}_1$-ring $R$. We will say that
  $M$ is {\tmdfn{flat}} if the following conditions are satisfied:
  \begin{enumerate}
    \item The homotopy group $\pi_0 M$ is flat as a left module over $\pi_0 R$
    in the usual sense.
    
    \item For each $n \in \mathbb{Z}$, the natural map $\tmop{Tor}_0^{\pi_0 R}
    (\pi_n R, \pi_0 M) \rightarrow \pi_n M$ is an isomorphism of abelian
    groups.
  \end{enumerate}
\end{definition}

\begin{definition}[{\cite[Definition~7.2.4.1]{Lurie2017}}]
  Let $R$ be an $\mathbb{E}_1$-ring. We let $\tmop{LMod}_R^{\tmop{perf}}$
  denote the smallest stable subcategory of $\tmop{LMod}_R$ which contains $R$
  (regarded as a left module over itself) and is closed under retracts. We
  will say that a left $R$-module $M$ is perfect if it belongs to
  $\tmop{LMod}_R^{\tmop{perf}}$.
\end{definition}

\begin{definition}[{\cite[Definition~7.2.4.8]{Lurie2017}}]
  Let $\mathcal{C}$ be a compactly generated $\infty$-category. We will say
  that an object $C \in \mathcal{C}$ is {\tmdfn{almost compact}} if
  $\tau_{\leq n} C$ is a compact object of $\tau_{\leq n}  \mathcal{C}$ for
  all $n \geq 0$.
\end{definition}

\begin{definition}[{\cite[Definition~7.2.4.10]{Lurie2017}}]
  \label{def:aperf}Let $R$ be a connective $\mathbb{E}_1$-ring. We will say
  that a left $R$-module $M$ is {\tmdfn{almost perfect}} if there exists an
  integer $k$ such that $M \in (\tmop{LMod}_R)_{\geq k}$ and is almost compact
  as an object of $(\tmop{LMod}_R)_{\geq k}$. We let
  $\tmop{LMod}_R^{\tmop{aperf}}$ denote the full subcategory of
  $\tmop{LMod}_R$ spanned by the almost perfect left $R$-modules.
\end{definition}

\begin{proposition}[{\cite[Proposition~7.2.4.11]{Lurie2017}}]
  \label{prop:aperf}Let $R$ be a connective $\mathbb{E}_1$-ring. Then:
  \begin{enumeratenumeric}
    \item The full subcategory $\tmop{LMod}_R^{\tmop{aperf}} \subseteq
    \tmop{LMod}_R$ is closed under translation and finite colimits, and is
    therefore a stable subcategory of $\tmop{LMod}_R$;
    
    \item The full subcategory $\tmop{LMod}_R^{\tmop{aperf}} \subseteq
    \tmop{LMod}_R$ is closed under retracts;
    
    \item Every perfect left $R$-module is almost perfect;
    
    \item The full subcategory $(\tmop{LMod}_R^{\tmop{aperf}})_{\geq 0}
    \subseteq \tmop{LMod}_R$ is closed under geometric realizations of
    simplicial objects;
    
    \item Let $M$ be a left $R$-module which is connective and almost perfect.
    Then $M$ can be obtained as the geometric realization of a simplicial left
    $R$-module $P_{\bullet}$ such that each $P_n$ is a free $R$-module of
    finite rank.
  \end{enumeratenumeric}
\end{proposition}

\begin{proposition}
  \label{prop:aperf-base-change}Let $f \of A \rightarrow A'$ be a map of
  connective $\mathbb{E}_1$-rings. Let $M$ be a connective left $A$-module and
  set $M' = A' \otimes_A M$. If $M$ is an almost perfect left $A$-module, then
  $M'$ is an almost perfect left $A'$-module.
\end{proposition}

\begin{proof}
  Since $M$ is connective and almost perfect, by Proposition~\ref{prop:aperf},
  there exists a simplicial object $P_{\bullet}$ in $\tmop{LMod}_A$ such that
  each $P_n$ is a free $A$-module of finite rank and $M$ is equivalent to the
  geometric realization of $P_{\bullet}$. Therefore $M'$ is equivalent to the
  geometric realization of $A' \otimes_A P_{\bullet}$, by the fact the tensor
  products commute with small colimits. On the other hand, each $A' \otimes_A
  P_n$ is a free $A'$-module of finite rank, hence perfect, thus almost
  perfect. Now $M'$ is equivalent to the geometric realization of almost
  perfect modules, therefore $M'$ is almost perfect by
  Proposition~\ref{prop:aperf}.
\end{proof}

\begin{definition}[{\cite[Definition~7.2.4.13]{Lurie2017}}]
  A discrete associative ring $R$ is {\tmdfn{left coherent}} if every finitely
  generated left ideal of $R$ is finitely presented as a left $R$-module.
\end{definition}

\begin{definition}[{\cite[Definition~7.2.4.16]{Lurie2017}}]
  \label{def:coherentE1}Let $R$ be an $\mathbb{E}_1$-ring. We will say that
  $R$ is {\tmdfn{left coherent}} if the following conditions are satisfied:
  \begin{enumeratenumeric}
    \item The $\mathbb{E}_1$-ring $R$ is connective;
    
    \item The discrete associative ring $\pi_0 R$ is left coherent;
    
    \item For each $n \geq 0$, the homotopy group $\pi_n R$ is finitely
    presented as a left module over $\pi_0 R.$
  \end{enumeratenumeric}
\end{definition}

\begin{proposition}[{\cite[Proposition~7.2.4.17]{Lurie2017}}]
  Let $R$ be an $\mathbb{E}_1$-ring and $M$ a left $R$-module. Suppose that
  $R$ is left coherent. Then $M$ is almost perfect if and only if the
  following conditions are satisfied:
  \begin{enumerateroman}
    \item For $m \ll 0$, $\pi_m M = 0$;
    
    \item For every integer $m$, $\pi_m M$ is finitely presented as a left
    $\pi_0 R$-module.
  \end{enumerateroman}
\end{proposition}

\begin{corollary}
  \label{cor:coherent_aperf}Let $R$ be a left coherent $\mathbb{E}_1$-ring,
  then H $\pi_0 (R)$ as a left $R$-module is almost perfect.
\end{corollary}

\subsection{Nilpotent, local and complete modules}

We will include several definitions and propositions from {\cite{Lurie2018}},
Chapter 7.

\begin{definition}[{\cite[Definition~7.1.1.1, Example~7.1.1.2]{Lurie2018}}]
  Let $R$ be a connective $\mathbb{E}_{\infty}$-ring and let $x \in \pi_0 R$.
  An $R$-module $M$ is {\tmdfn{$x$-nilpotent}} if the localization $M [1 / x]$
  vanishes. Equivalently, $M$ is $x$-nilpotent if and only if the action of
  $x$ on $\pi_{\ast} M$ is locally nilpotent, that is, if and only if for each
  $y \in \pi_j M$, there exists an integer $n \gg 0$ such that $x^n y = 0$ in
  $\pi_j M$ for all $j \in \mathbb{Z}$.
\end{definition}

\begin{definition}[{\cite[Definition~7.1.1.6]{Lurie2018}}]
  Let $R$ be a connective $\mathbb{E}_{\infty}$-ring and let $I \subseteq
  \pi_0 R$ be an ideal. We say that an $R$-module $M$ is
  {\tmdfn{$I$-nilpotent}} if it is $x$-nilpotent for each $x \in I$.
\end{definition}

\begin{definition}[{\cite[Definition~7.2.4.1]{Lurie2018}}]
  Let $R$ be a connective $\mathbb{E}_{\infty}$-ring and let $I \subseteq
  \pi_0 R$ be an ideal. We say that an $R$-module $M$ is {\tmdfn{$I$-local}}
  if for every $I$-nilpotent $R$-module $N$, the mapping space
  $\tmop{Map}_{\tmop{Mod}_R} (N, M)$ is contractible.
\end{definition}

\begin{definition}[{\cite[Definition~7.3.1.1]{Lurie2018}}]
  Let $R$ be a connective $\mathbb{E}_{\infty}$-ring and let $I \subseteq
  \pi_0 R$ be an ideal. We will say that an $R$-module $M$ is
  {\tmdfn{$I$-complete}} if for every $I$-local $R$-module $N$, the mapping
  space $\tmop{Map}_{\tmop{Mod}_R} (N, M)$ is contractible.
\end{definition}

\begin{corollary}
  \label{cor:nil-compl}Let $R$ be a connective $\mathbb{E}_{\infty}$-ring and
  let $I \subseteq \pi_0 R$ be an ideal. If $M$ is an $I$-nilpotent
  $R$-module, then it is also an $I$-complete $R$-module.
\end{corollary}

\begin{proposition}[{\cite[Proposition~7.3.1.4 and
Notation~7.3.1.5]{Lurie2018}}]
  Let $R$ be a connective $\mathbb{E}_{\infty}$-ring and let $I \subseteq
  \pi_0 R$ be a finitely generated ideal. Then every left $R$-module $M$ fits
  into an (essentially unique) fiber sequence $M' \rightarrow M \rightarrow
  M''$, where $M'$ is $I$-local and $M''$ is $I$-complete. Moreover, there is
  a functor, called the $I$-completion functor, $\tmop{Mod}_R \rightarrow
  \tmop{Mod}_R$, which maps $M$ to $M''$. We denote by $M_I^{\wedge}$ the
  image of $M$ under the $I$-completion functor.
\end{proposition}

We can compute the $I$-completion functor when $I$ is principal:

\begin{proposition}[{\cite[Proposition~7.3.2.1]{Lurie2018}}]
  \label{prop:calc-compl}Let $R$ be a connective $\mathbb{E}_{\infty}$-ring
  and let $x \in \pi_0 R$ be an element. For any $R$-module $M \in
  \tmop{Mod}_R$, let $T (M)$ denote the limit of the tower
  \[ \cdots \xrightarrow{x} M \xrightarrow{x} M \xrightarrow{x} M
     \xrightarrow{x} M \]
  Then $T (M)$ is $(x)$-local and the $(x)$-completion of $M$ can be
  identified with the cofiber of the canonical map $\theta \of T (M)
  \rightarrow M$.
\end{proposition}

\begin{corollary}[{\cite[Corollary~7.3.2.2]{Lurie2018}}]
  Let $R$ be a connective $\mathbb{E}_{\infty}$-ring and let $x \in \pi_0 R$
  be an element. The following conditions on an $R$-module $M \in
  \tmop{Mod}_R$ are equivalent:
  \begin{enumeratenumeric}
    \item The module $M$ is $(x)$-complete.
    
    \item The limit of the tower
    \[ \cdots \xrightarrow{x} M \xrightarrow{x} M \xrightarrow{x} M
       \xrightarrow{x} M \]
    vanishes.
  \end{enumeratenumeric}
\end{corollary}

\begin{corollary}[{\cite[Corollary~7.3.2.3]{Lurie2018}}]
  Let $R$ be a connective $\mathbb{E}_{\infty}$-ring, $I \subseteq \pi_0 R$ an
  ideal and $x \in \pi_0 R$ an element. Then the $(x)$-completion functor
  $\tmop{Mod}_R \rightarrow \tmop{Mod}_R, M \mapsto M_{(x)}^{\wedge}$ carries
  $I$-complete modules to $I$-complete modules.
\end{corollary}

\begin{corollary}[{\cite[Corollary~7.3.2.4]{Lurie2018}}]
  \label{cor:conn-compl}Let $R$ be a connective $\mathbb{E}_{\infty}$-ring, $x
  \in \pi_0 R$ and let $M$ be an $R$-module.
  \begin{enumeratenumeric}
    \item If the $R$-module $M$ is connective, then the $(x)$-completion
    $M_{(x)}^{\wedge}$ is connective.
    
    \item If $M \in (\tmop{Mod}_R)_{\leq 0}$, then $M_{(x)}^{\wedge} \in
    (\tmop{Mod}_R)_{\leq 1}$.
  \end{enumeratenumeric}
\end{corollary}

\begin{proof}
  Let $T (M)$ be the limit of the tower $\left( \cdots \xrightarrow{x} M
  \xrightarrow{x} M \xrightarrow{x} M \right)$. Then by
  Proposition~\ref{prop:calc-compl}, we have the cofiber sequence $T (M)
  \rightarrow M \rightarrow M_{(x)}^{\wedge}$ which gives rise to a long exact
  sequence
  \[ \cdots \rightarrow \pi_n (T (M)) \rightarrow \pi_n (M) \rightarrow \pi_n
     (M_{(x)}^{\wedge}) \rightarrow \pi_{n - 1} (T (M)) \rightarrow \pi_{n -
     1} (M) \rightarrow \pi_{n - 1} (M_{(x)}^{\wedge}) \rightarrow \cdots \]
  Furthermore, let $T_n (M)_{\ast}$ be the tower
  \[ \cdots \xrightarrow{x} \pi_n (M) \xrightarrow{x} \pi_n (M)
     \xrightarrow{x} \pi_n (M) \]
  Then there is a Milnor sequence
  \[ 0 \rightarrow \lim^1 T_{n + 1} (M)_{\ast} \rightarrow \pi_n (T (M))
     \rightarrow \lim T_n (M)_{\ast} \rightarrow 0 \]
  Especially, if $M$ is assumed to be connective, then $T_n (M)_{\ast}$ is a
  tower of $0$ for $n < 0$, which implies that $\pi_{n - 1} (T (M))$ vanishes
  when $n < 0$. We deduce from the long exact sequence that $\pi_n
  (M_{(x)}^{\wedge})$ vanishes when $n < 0$. Similarly, if $M \in
  (\tmop{Mod}_R)_{\leq 0}$, then $T_n (M)_{\ast}$ is a tower of $0$ for $n >
  0$, thus $\pi_n (T (M))$ vanishes when $n \geq 0$. We deduce from the long
  exact sequence that $\pi_n (M_{(x)}^{\wedge})$ vanishes when $n > 0$.
\end{proof}

\begin{proposition}[{\cite[Corollary~7.3.3.3]{Lurie2018}}]
  \label{prop:compl-gens}Let $R$ be a connective $\mathbb{E}_{\infty}$-ring
  and $I \subseteq \pi_0 R$ be a finitely generated ideal. Let $M$ be an
  $R$-module. Then the following conditions on $M$ are equivalent:
  \begin{enumeratenumeric}
    \item $M$ is $I$-complete;
    
    \item For each $x \in I$, $M$ is $(x)$-complete;
    
    \item There exists a set of generators $x_1, \ldots, x_n$ for the ideal
    $I$ such that $M$ is $(x_i)$-complete for $i = 1, \ldots, n$.
  \end{enumeratenumeric}
\end{proposition}

\begin{remark}[{\cite[Corollary~7.3.3.6]{Lurie2018}}]
  \label{rem:complete-base-indep}Let $\phi \of R \rightarrow R'$ be a morphism
  of connective $\mathbb{E}_{\infty}$-rings, $I \subseteq \pi_0 R$ a finitely
  generated ideal and $I' = \phi (I) \pi_0 (R')$ the ideal generated by the
  image of $I$. Then
  \begin{enumeratenumeric}
    \item An $R'$-module $M$ is $I'$-complete if and only if it is
    $I$-complete as an $R$-module;
    
    \item For every $R'$-module $M$, the canonical map $M \rightarrow
    M_{I'}^{\wedge}$ exhibits $M_{I'}^{\wedge}$ as an $I$-completion of $M$,
    regarded as a morphism of $R$-modules.
  \end{enumeratenumeric}
\end{remark}

\begin{theorem}[{\cite[Theorem~7.3.4.1]{Lurie2018}}]
  \label{thm:cplhomot}Let $R$ be an $\mathbb{E}_{\infty}$-ring, let $I
  \subseteq \pi_0 R$ be a finitely generated ideal and let $M$ be an
  $R$-module. The following conditions are equivalent:
  \begin{enumeratealpha}
    \item The $R$-module $M$ is $I$-complete;
    
    \item For every integer $k$, the homotopy group $\pi_k M$ satisfies the
    condition that for each $x \in I$, we have $\tmop{Ext}_A^0 (A [1 / x],
    \pi_k M) = 0 = \tmop{Ext}_A^1 (A [1 / x], \pi_k M)$ where $A = \pi_0 R.$ 
  \end{enumeratealpha}
\end{theorem}

\begin{proposition}[{\cite[Proposition~7.3.4.8]{Lurie2018}}]
  Let $R$ be a connective $\mathbb{E}_{\infty}$-ring, let $I \subseteq \pi_0
  R$ be a finitely generated ideal, and let $x \in \pi_0 R$ be an element
  whose image in $(\pi_0 R) / I$ is invertible. If $M$ is an $I$-complete left
  $R$-module, then multiplication by $x$ induces an equivalence from $M$ to
  itself.
\end{proposition}

\begin{proposition}[{\cite[Proposition~7.3.5.7]{Lurie2018}}]
  \label{prop:aperf-complete}Let $R$ be a connective
  $\mathbb{E}_{\infty}$-ring, let $I \subseteq \pi_0 R$ be a finitely
  generated ideal, and let $M$ be an almost perfect $R$-module. If $R$ is
  $I$-complete, then so is $M$.
\end{proposition}

\begin{definition}[{\cite[Section~7.3.6]{Lurie2018}}]
  Let $R$ be a discrete commutative ring and $M$ be a (discrete) $R$-module.
  Let $I \subseteq R$ be a finitely generated ideal. The {\tmdfn{$I$-adic
  completion}} of $M$, denoted by $\tmop{Cpl} (M ; I)$, is defined to be the
  limit of tower $\lim M / I^n M$. There is a canonical map $M \rightarrow
  \tmop{Cpl} (M ; I)$. An $R$-module $M$ is called {\tmdfn{$I$-adically
  complete}} if the canonical map $M \rightarrow \tmop{Cpl} (M ; I)$ is an
  isomorphism, and $M$ is called $I$-adically separated if the canonical map
  $M \rightarrow \tmop{Cpl} (M ; I)$ is injective.
\end{definition}

\begin{proposition}[{\cite[Corollary~7.3.6.3]{Lurie2018}}]
  \label{prop:I-adic}Let $R$ be a discrete commutative ring, let $I \subseteq
  R$ be a finitely generated ideal, and let $M$ be a discrete $R$-module. The
  following conditions are equivalent:
  \begin{enumeratealpha}
    \item The module $M$ is $I$-adically complete;
    
    \item The module $H M$ is $I$-complete and $M$ is $I$-adically separated.
  \end{enumeratealpha}
\end{proposition}

\begin{warning}
  By Proposition~\ref{prop:I-adic}, the concept of $I$-adic completeness does
  not coincide with the concept of $I$-completeness for discrete modules over
  discrete commutative rings. Rather, the former is stronger than the latter.
\end{warning}

\begin{definition}
  A spectrum $X$ is called {\tmdfn{$p$-complete}} if it is $(p)$-complete as
  an $\mathbb{S}$-module. For any spectrum $X$, the $p$-completion of $X$,
  denoted by $X_p^{\wedge}$, is the $(p)$-completion of $X$ as an
  $\mathbb{S}$-module.
\end{definition}

\begin{remark}
  When $M$ is an $R$-module for a connective $\mathbb{E}_{\infty}$-ring $R$,
  $(p)$ is also an ideal of $\pi_0 R$. In this case, it follows from
  Remark~\ref{rem:complete-base-indep} that $M$ is $(p)$-complete as an
  $\mathbb{S}$-module if and only if it is $(p)$-complete as an $R$-module, so
  there is completely no ambiguity to talk about $p$-completeness. Similarly,
  Remark~\ref{rem:complete-base-indep} implies the $p$-completion of an
  $R$-module $M$ is the underlying spectrum of the $(p)$-completion of $M$ as
  an $R$-module.
\end{remark}

\begin{corollary}
  \label{cor:compl-mod-p}Let $X$ be a bounded below spectrum. If $X$ is
  $p$-complete and $H\mathbb{F}_p \otimes X \simeq 0$, then $X \simeq 0$.
\end{corollary}

\begin{proof}
  We will show inductively on $n \in \mathbb{Z}$ that $\pi_n X = 0$.
  \begin{enumeratenumeric}
    \item Since $X$ is bounded below, $\pi_n X = 0$ for $n \ll 0$.
    
    \item Suppose now that for every $m < n$, we have $\pi_m X = 0$. We will
    show that $\pi_n X = 0$. In this case, we have $0 = \pi_n (H\mathbb{F}_p
    \otimes X) \cong \tmop{Tor}_0^{\mathbb{Z}} (\mathbb{F}_p, \pi_n X)$. Thus
    for each $x \in \pi_n X$, there exists (by axiom of choice) a sequence
    $(x_j)_{j \in \mathbb{N}} \in (\pi_n X)^{\mathbb{N}}$ such that $x_0 = x$
    and $x_j = px_{j + 1}$ for all $j \in \mathbb{N}$, which gives rise to a
    map $\varphi_x \of \mathbb{Z} [1 / p] \rightarrow \pi_n X$ of abelian
    groups given by $\varphi (1 / p^j) = x_j$. Theorem~\ref{thm:cplhomot}
    tells us that $\varphi_x = 0$, and especially, $x = 0$. In conclusion, we
    have proved that $x = 0$ for each $x \in \pi_n X$, thus $\pi_n X = 0$.
  \end{enumeratenumeric}
\end{proof}

\end{document}